\documentclass[11pt]{article}\usepackage[]{graphicx}\usepackage[]{color}
\makeatletter
\def\maxwidth{ %
  \ifdim\Gin@nat@width>\linewidth
    \linewidth
  \else
    \Gin@nat@width
  \fi
}
\makeatother

\definecolor{fgcolor}{rgb}{0.345, 0.345, 0.345}

\usepackage{framed}
\makeatletter
 {\par\unskip\endMakeFramed%
 \at@end@of@kframe}
\makeatother

\definecolor{shadecolor}{rgb}{.97, .97, .97}
\definecolor{messagecolor}{rgb}{0, 0, 0}
\definecolor{warningcolor}{rgb}{1, 0, 1}
\definecolor{errorcolor}{rgb}{1, 0, 0}

\usepackage{alltt}
\usepackage{geometry,latexsym,graphics,authblk,color,amsmath,amsfonts,amssymb,amsbsy,amsthm,natbib,enumerate,url,multirow,threeparttable}
\usepackage[colorlinks,
            linkcolor=blue,
            anchorcolor=blue,
            citecolor=blue]{hyperref}
\usepackage{float}
\usepackage{algorithm,algpseudocode}
\usepackage{tikz}
\usepackage{tcolorbox}
\usepackage{setspace}
\newtheorem{thm}{Theorem}
\newtheorem{lem}{Lemma}

\newtheorem{assumption}{Assumption}

\newtheorem{defi}{Definition}

\newcommand\R{{\sf I\kern-0.1em R}}

\newcommand{\mathsym}[1]{{}}

\IfFileExists{upquote.sty}{\usepackage{upquote}}{}
\begin{document}
\title{Lyapunov function approach for approximation algorithm design and analysis: with applications in continuous submodular maximization}
\author[]{Donglei Du\thanks{\tt ddu@unb.ca}}
\affil[]{Faculty of Management, University of New Brunswick, Fredericton, New Brunswick, Canada, E3B 5A3}

\maketitle

\begin{abstract} We propose a two-phase systematical framework for approximation algorithm design and analysis via Lyapunov function. The first phase consists of using Lyapunov function as an input and outputs a continuous-time  approximation algorithm with a provable approximation ratio. The second phase then converts this continuous-time algorithm to a discrete-time algorithm with almost the same approximation ratio along with provable time complexity. One distinctive feature of our framework is that we only need to know the parametric form of the Lyapunov function whose complete specification will not be decided until the end of the first phase by maximizing the approximation ratio of the continuous-time algorithm. Some immediate benefits of the Lyapunov function approach include: (i) unifying many existing algorithms; (ii) providing a guideline to design and analyze new algorithms; and (iii) offering new perspectives to potentially improve existing algorithms. We use various submodular maximization problems as running examples to illustrate our framework.
\end{abstract}
 

\section{Introduction}

More often than not, curious readers find themselves spending much time wondering where those nice algorithms, analyses and proofs come from in the first place. This work provides a systematic framework to attempt to shed some light on the design and analysis of approximation algorithm. We will illustrate this framework with representative examples from submodular maximization. 

We propose using Lyapunov function not only as a technique for analysis and proof as have been done in the previous literature, but also as a guideline for design of approximation algorithm. The difference between \emph{design} and \emph{analysis} is that the former explains how to systematically design an algorithm and analysis is just a consequence of the design process, while the latter assumes an algorithm is known beforehand and then analyzes its approximation ratio and time complexity afterwards. 

One major difference between this work and previous Lyapunov function based work is that almost all previous literature has been using Lyapunov function or potential function as an analysis tool and a proof technique~\citep{wilson2018lyapunov,bansal2019potential,chen2021unified} after the Lyapunov function is given. However, the most important question of where these Lyapunov functions or potential functions are pulled is left unexplained or explained in a qualitative, and highly ad-hoc manner. This work therefore is an attempt to answer to the ``a potential function appearing out of nowhere." phenomenon~\cite{buchbinder2009design}).

Lyapunov function in the continuous-time setting is also called potential function in discrete-time setting. In this work, we distinguish these two terms by reserving Lyapunov function for the continuous-time setting and potential function for the discrete-time setting. The title of this article emphasizes our preference to the continuous-time setting as the focal point. 

We present a two-phase framework for approximation algorithm design and analysis. In the first phase, consisting of four steps, we use Lyapunov function as a guideline to design continuous-time algorithm with a provable approximation ratio. In the second phase, consisting of two steps, we convert the continuous-time algorithm to a discrete-time algorithm with almost the same approximation ratio along with provable time complexity. 

Approximation ratio and time complexity are the two main objectives in any approximation algorithm design and analysis. Our two-phase framework decouples these two jobs. This sequential approach is dictated by the usual dominant priority on approximation ratio over time complexity. The first phase focuses on identifying and proving the approximation ratio in the continuous-time domain and the second phase takes on the time complexity issue in the discrete-time domain. Time complexity in the discrete-time setting consists of two ingredients: the number of oracle calls per iteration and the number of iterations. In our framework, we always guarantee the first factor is controlled in a polynomial number of oracle calls at each instant in the first phase. Therefore, we use the term ``time complexity'' in the discrete-time setting to refer only to the number of iterations.


The disintegration of the continuous-time and discrete-time settings allows us to leverage two large areas of existing knowledge. In the continuous-time domain, we can exploit more powerful tools, such as calculus, functional analysis, and differential equation, etc., resulting in much easier and cleaner mathematical treatments than their discrete-time counterparts. In the discrete-time domain, we can utilize another huge literature on numerical methods to convert continuous-time algorithm to discrete-time algorithm. The disintegration also allows us to pin down the factors that are respectively responsible for approximation ratio and time complexity and hence the algorithm design becomes more transparent. 

As a comparison, an alternative way for designing a discrete-time algorithm can skip the continuous-time phase and directly operates in the discrete-time domain to handle approximation ratio and time complexity simultaneously. While all the aforementioned decoupling benefits disappear in the direct method, sometimes there are also some advantages as to be explained in Section~\ref{eq:subsubsec:direct-method}.

\section{Relevant literature}\label{sec:literature-review}

Lyapunov function was originally used to study the stability of equilibria of an ODE, and hence has been widely used in stability theory of dynamical systems and control theory~\citep{Lyapunov1992general}. Systematic treatments of Lyapunov function approach have also been adopted in convex optimization field for convergence rate analysis, including both as a proof technique~(e.g., \cite{nemirovskij1983problem,polyak1987introduction,bansal2019potential}) and as an algorithm design technique (e.g., \cite{diakonikolas2019approximate}). 

\cite{bansal2019potential} focus on how to use Lyapunov function as a proof technique in convergence rate analysis of many first-order methods, but offer no insight on algorithm design. \cite{chen2021unified} present a unified convergence analysis for many first-order convex optimization methods via the strong Lyapunov function. \cite{diakonikolas2019approximate} is closely related to our work, as they propose a general technique for both convergence analysis and design of algorithms of many first-order methods in convex optimization. However, our focus here is on the design of approximation algorithm in non-convex optimization rather than exact methods in convex optimization. There is also a large body of literature surrounding Nesterov's accelerated method in which Lyapunov function based approaches are also utilized in their analysis, such as \cite{su2016differential,wibisono2016variational,wilson2018lyapunov,chen2021unified,wilson2021Lyapunov}, among many others.

Potential function based methods, as a proof technique in approximation algorithm analysis, have also been used in areas beyond convex optimization, e.g., \cite{buchbinder2015tight,badanidiyuru2020submodular,cohen2022improved}, etc., to name just a few.

The rest of this paper is organized as follows. In Section~\ref{subsec:algo-design-L-fun}, we describe the first phase of our framework on algorithm design via Lyapunov function in the continuous-time setting and illustrate the framework by examples from submodular maximization. In Section~\ref{sec:algo-design-L-function-dis}, we explain the second phase of the framework via potential function in the discrete-time setting and illustrate the framework with the same examples. Concluding remarks are provided in Section~\ref{sec:conclusion}

\section{Algorithm design via Lyapunov function: continuous-time}\label{subsec:algo-design-L-fun}

In Section~\ref{sec:Lyapunov_intro}, we introduce the concept of Lyapunov function of an algorithm in the continuous-time setting. In Section~\ref{subsec:algo-design-con-time}, we present a four-step process on how to use Lyapunov function to systematically design approximation algorithm. In Sections~\ref{subsec:mono-exp1}, \ref{subsec:mono-exp2}, and \ref{subsec:con-exp3}, we illustrate the first phase via examples from submodular maximization, which has been an active area of research (see, e.g., the survey by~\cite{buchbinder2018submodular}). 

Denote $\mathbb{R}^n$ as the $n$-dimension Euclidean space along with its Euclidean norm $\|x\|_2=\sqrt{\sum_{i=1}^n}x_i^2$, the infinite norm $\|x\|_{\infty}=\max_{i=1}^nx_i$, and inner product $\langle x, y\rangle=\sum_{i=1}^n x_iy_i$, $\forall x, y\in \mathbb{R}^n$. Denote $\mathbb{R}^n_{+}$ and $\mathbb{R}^n_{>0}$ as the nonnegative and positive vectors in $\mathbb{R}^n$, respectively. Denote $x^{+}=\max\{x, 0\}, \forall x\in\mathbb{R}$. We use boldface letters for vectors in $\mathbb{R}^n$. We use the ``dot" notation and the ``prime mark" interchangeably for time derivative, for example $\dot{x}(t)=(x(t))^{\prime}=\frac{dx}{dt}$. We use $\nabla$ for the gradient of a multivariate function.

\subsection{Lyapunov function}\label{sec:Lyapunov_intro}
Consider the following generic maximization problem
\begin{equation}\label{eq:generic_max_prob}
\max\limits_{\mathbf{x}\in C} F(\mathbf{x}),
\end{equation}
where $C\subseteq\mathbb{R}^n$ is a feasible set and $\mathbb{R}^n\ni \mathbf{x}\mapsto F(\mathbf{x})\in \mathbb{R}$ is a real-valued function. Let $\mathbf{x}^*$ be an optimal solution of (\ref{eq:generic_max_prob}):
\[
\mathbf{x}^*\in \operatorname*{argmax}\limits_{\mathbf{x}\in C} F(\mathbf{x}).
\]

\begin{defi}\label{def:algo-con} \emph{(Continuous-time algorithm)} A (continuous-time) algorithm for the problem (\ref{eq:generic_max_prob}) is a vector-valued function $\mathbf{x}(t)$:
\[
[0, T]\ni t\mapsto \mathbf{x}(t)\in\mathbb{R}^n, 
\]
starting from $\mathbf{x}(0)\in C$ at time 0 and outputting $\mathbf{x}(T)\in C$ at time $T$. 
\qed
\end{defi}

The above definition is restrictive on the output. A more general case can output a solution $\xi(\mathbf{x}(t))\in C$, where $\xi$ is a map which takes the $\mathbf{x}(t)\in C, \forall t\in [0, T]$, as an input and outputs a feasible solution. There is no need to even require $\mathbf{x}(t)\in C, \forall t\in [0, T]$, although in this paper we focus on algorithms satisfying $\mathbf{x}(t)\in C, \forall t\in [0, T]$. The most common choice is the one in the above definition $\xi(\mathbf{x}(t))=x(T)$, which will be used throughout this paper. Some other possibilities exist, such as $\xi(\mathbf{x}(t))=\int_0^T x(t)d\mu(t)$, which is feasible when $C$ is convex and $\mu$ is a probability measure.

An algorithm $\mathbf{x}(t)$ can often be described as a dynamical system:
\[
\dot{\mathbf{x}}(t)=\boldsymbol{\varphi}(\mathbf{x}(t)),
\]
where 
\[
\mathbb{R}^n\ni \mathbf{x}(t)\mapsto \boldsymbol{\varphi}(\mathbf{x}(t))\in\mathbb{R}^n.
\]

\begin{defi}\label{def:Lyapunov-fun} \emph{(Continuous-time Lyapunov function)} A function $[0, T]\ni t\mapsto E(\mathbf{x}(t))\in\mathbb{R}$ is a (continuous-time) \emph{Lyapunov function} associated with an algorithm $\mathbf{x}(t)$ in Definition~\ref{def:algo-con} if it is non-decreasing in $t\in [0, T]$.\qed
\end{defi}

Lyapunov function can take on many different parametric forms depending on the problems on hand. In approximation algorithm design and analysis, the following parametric form goes a long way in unifying many algorithms:
\begin{equation}\label{eq:Lyapunov-function-continuous-time}
E(\mathbf{x}(t))=a_tF(\mathbf{x}(t))-b_tF(\mathbf{x}^*),
\end{equation}
where $[0, T]\ni t\mapsto a_t, b_t\in \mathbb{R}_{+}$ are two functions belonging to certain problem-specific class $Q$. In the remaining of the article, we make the following assumptions on $a$ and $b$. 
\begin{assumption}\label{assume:q-b} Functions $[0, T]\ni t\mapsto a_t, b_t\in \mathbb{R}_{>}$ are non-negative, non-decreasing, and differentiable; namely, they belong to the following class of functions:
\begin{equation*}
Q=\{(a, b): a_0>0, \dot{a}_t\ge 0; b_0\ge 0, \dot{b}_t\ge 0\},
\end{equation*}\qed
\end{assumption}

The connection between the Lyapunov function of an algorithm and the approximation ratio of the algorithm is as follows. Given a Lyapunov function $E(\mathbf{x}(t))$ of an algorithm $\mathbf{x}(t)$, since $E$ is increasing in $t$, we have
\[
E(\mathbf{x}(T))=a_T F(\mathbf{x}(T))-b_T F(\mathbf{x}^*)\ge E(0)=a_0 F(\mathbf{x}(0))-b_0 F(\mathbf{x}^*),
\]
implying that the approximation ratio of the algorithm $\mathbf{x}(t)$ which outputs $\mathbf{x}(T)\in C$ is
\begin{equation}\label{eq:appro_ratio}
f(\mathbf{x}(T))\ge \frac{b_T-b_0}{a_T}F(\mathbf{x}^*)+\frac{a_0 }{a_T}F(\mathbf{x}(0))\ge \frac{b_T-b_0}{a_T}F(\mathbf{x}^*),
\end{equation}
where the last inequality holds whenever $F$ is nonnegative.

\subsection{Algorithm design: first phase}\label{subsec:algo-design-con-time}
Now we present the first phase of our algorithm design and analysis framework. This phase takes the Lyapunov function $E(\mathbf{x}(t))$ as an input and outputs an approximation algorithm with a provable approximation ratio. 

One distinctive feature of our framework is that we only need to know the parametric form of $E(\mathbf{x}(t))$ and leave the two function $a_t, b_t$ and the stopping time $T$ unspecified until the end to choose them so as to maximize the approximation ratio of the algorithm.

The first phase of the framework consists of four steps and we will present them in such a way that every step is a sufficient condition of the previous step. When proceeding forward, it explains how an algorithm is produced. When reading backward, it provides a formal proof of the approximation ratio of the produced algorithm.

\begin{description}
\item [Step 1. (Specify the parametric form of the Lyapunov function)] Pre-specifying the parametric form of Lyapunov function for a given problem is a prerequisite for our framework. Below we assume the Lyapunov function $E(\mathbf{x}(t))$ takes the parametric form in (\ref{eq:Lyapunov-function-continuous-time}) for some unknown algorithm $x(t), t\in [0, T]$. However, the basic principle applies to other parametric forms of Lyapunov function which will be discussed further in Section~\ref{sec:conclusion}. 

\item [Step 2. (Bound the optimal value)] A sufficient condition for $\mathbf{x}(t)$ to have an approximation ratio specified in (\ref{eq:appro_ratio}) is that $E(\mathbf{x}(t))$ is non-decreasing, which is equivalent to the non-negativeness of its  derivative:
\begin{equation*}
\dot{E}(\mathbf{x}(t))=\dot{a}_t F(\mathbf{x}(t))+a_t\langle\nabla F(\mathbf{x}(t)), \dot{\mathbf{x}}(t) \rangle -\dot{b}_t F(\mathbf{x}^*)\ge 0,
\end{equation*}
assuming all functions involved are differentiable. For non-smooth $F$, sub-gradients, if well-defined and exist, replace the role of gradients (e.g, the sub-gradient for convex function~\cite{rockafellar1970convex}, and the Clarke's generalized gradient for locally Lipschitz continuous functions~\cite{clarke1990optimization}, among others). 

Because Algorithm $\mathbf{x}(t)$ cannot use any information on $F(\mathbf{x}^*)$, one sufficient condition for $\dot{E}(\mathbf{x}(t))\ge 0$ is to replace the term $F(\mathbf{x}^*)$ therein with an upper bound $U_t\ge F(\mathbf{x}^*)$ such that $U_t$ does not depend on the optimal solution $\mathbf{x}^*$ (may depend on $\mathbf{x}(s), s\in [0, t]$), and then show that the following lower bound on $\dot{E}(\mathbf{x}(t))$ is non-negative:
\begin{equation}\label{eq:L-fun-G}
\dot{E}(\mathbf{x}(t))\overset{F(\mathbf{x}^*)\le U_t}{\ge} \dot{a}_t F(\mathbf{x}(t))+a_t\langle\nabla F(\mathbf{x}(t)), \dot{\mathbf{x}}(t) \rangle -\dot{b}_t U_t\ge 0
\end{equation}

Now the critical part is to find a useful upper bound $U_t$ which will be problem-specific. However for a large class of problems, first-order algorithms (only function values and its gradients are involved) can be employed, in which cases, an upper bound $U_t$ normally takes on the following parametric form:
\begin{equation}\label{eq:U_t-form}
U_t=c(\mathbf{x}(t))F(\mathbf{x}(t))+d(t)\langle \nabla F(\mathbf{x}(t), \mathbf{u}(\mathbf{x}(t)) \rangle
\end{equation}
where $[0, T]\ni t\mapsto c(\mathbf{x}(t))\in \mathbb{R}_{>0}$ is a positive scalar functions of $\mathbf{x}(t)$  and $\mathbb{R}^n\ni \mathbf{x}(t)\mapsto \mathbf{u}(\mathbf{x}(t))\in \mathbb{R}^n$ is a vector-valued function of $\mathbf{x}(t)$ such that its norm is no more than the the diameter of $C$: 
\begin{equation}\label{eqn:diamter}
\|\mathbf{u}(\mathbf{x}(t))\|^2\le D=\max_{x, y\in C}\|x-y\|^2. 
\end{equation}
Note that the functions $c(\cdot)$ and $d(\cdot)$ may have extra constraints related to the algorithm $\mathbf{x}(t)$. We collect all these constraints into a set $O$. 

With the upper bound $U_t$ in (\ref{eq:U_t-form}), then (\ref{eq:L-fun-G}) becomes
\begin{eqnarray}
\nonumber\dot{E}(\mathbf{x}(t))&\ge& \dot{a}_t F(\mathbf{x}(t))+a_t\langle\nabla F(\mathbf{x}(t)), \dot{\mathbf{x}}(t) \rangle -\dot{b}_t \left[c(\mathbf{x}(t))F(\mathbf{x}(t))+\langle \nabla F(X(t), d(t)\mathbf{u}(\mathbf{x}(t)) \rangle\right]\\
\label{eq:bound-F-star}&=&[\dot{a}_t-\dot{b}_tc(\mathbf{x}(t))]F(\mathbf{x}(t))+\langle\nabla F(\mathbf{x}(t)), a_t\dot{\mathbf{x}}(t)-\dot{b}_t d(t)\mathbf{u}(\mathbf{x}(t)) \rangle\ge 0
\end{eqnarray}

\item[Step 3. (Produce an algorithm)] Assume that $F$ is non-negative from now on.  From (\ref{eq:bound-F-star}), if we guarantee that the first term is nonnegative and second term is zero, then one sufficient condition for $\dot{E}(\mathbf{x}(t))\ge 0$ is as follows: 
\begin{eqnarray}
\label{eq:appro-ratio}\dot{a}_t&\ge &\dot{b}_tc(\mathbf{x}(t))\\
\label{eq:appro-alg}\dot{\mathbf{x}}(t)&=&\frac{\dot{b}_t}{a_t}d(t)\mathbf{u}(\mathbf{x}(t))\\
\label{eq:appro-O} c(\mathbf{x}(t))&\in& O
\end{eqnarray}

The ordinary differential equation (ode) in (\ref{eq:appro-alg}) is exactly the algorithm we have been seeking. However, to make sure this algorithm is valid (namely, (\ref{eq:appro-ratio})-(\ref{eq:appro-O}) is a feasible system), and indeed outputs a feasible solution (namely, $\mathbf{x}(T)\in C$), extra constraints on the two parametric functions $a_t$, $b_t$ and the termination time $T$ may arise. We collect all these constraints into a set 
\begin{equation}\label{eq:P}
P=\{(a, b, T): (\ref{eq:appro-ratio})-(\ref{eq:appro-O}) \text{ feasible system},\ \mathbf{x}(T)\in C\}.
\end{equation}
\item[Step 4. (Analyze the approximation ratio)] Finally, with our algorithm specified by the ode in (\ref{eq:appro-alg}), 
the best approximation ratio of this algorithm is to choose the two functions $a_t$, $b_t$, and the termination time $T$ to maximize the approximation ratio subject to (\ref{eq:appro-ratio})-(\ref{eq:appro-O}), the extra constraints $P$ in (\ref{eq:P}) from the feasibility of the algorithm, and any assumptions on $(a, b)$ which are collected in $Q$ as defined in Assumption~\ref{assume:q-b}:
\begin{eqnarray*}
&\sup_{a, b, T}& \frac{b_T-b_0}{a_T}\\
&\text{s.t.}&(a, b)\in Q\\
&& (a, b, T)\in P
\end{eqnarray*}
The last problem is a variational problem and sometimes we can find closed-form solutions in applications. 
\qed
\end{description}

One power of the proposed framework is to unify existing results and offer new perspectives on known results. In the next section we will present a few representative examples from the field of submodular maximization to illustrate this power. However, we believe the more promising power of this framework lies in developing improved algorithms for old problems and designing algorithms for new problems.  

We will use the following submodular maximization problem along with its variants as running examples to illustrate our framework, although the framework is applicable to any optimization problem in principle.
\begin{equation}\label{eq:sub_max_prob}
\max\limits_{\mathbf{x}\in C} F(\mathbf{x}).
\end{equation}
\begin{assumption}\label{assum:sub-oracle}
In problem (\ref{eq:sub_max_prob}), we make the following assumptions:
\begin{enumerate}[(1)]
\item $C\subseteq [0,1]^n$ is a convex set with diameter $D$ as in (\ref{eqn:diamter}), and $\mathbf{0}=(0,\ldots, 0)\in C$. 
\item $F:[0,1]^n\mapsto \mathbb{R}^{+}$ is a non-negative differentiable DR-submodular function~\citep{bian2017continuous}: $\forall \mathbf{x}\ge \mathbf{y}\in [0,1]^n; \forall \ell >0: \mathbf{x}+\ell \mathbf{e}_i, \mathbf{y}+\ell \mathbf{e}_i\in [0,1]^n$,
\[
F(\mathbf{x}+\ell \mathbf{e}_i)-F(\mathbf{x})\le F(\mathbf{y}+\ell \mathbf{e}_i)-F(\mathbf{y}), \forall i=1,\ldots, n
\]
\[
\bigg\Updownarrow
\]
\[
\nabla F(\mathbf{x})\le \nabla F(\mathbf{y}).
\]
\item  Three oracles are available: (i) computing the function value $F(\mathbf{x})$ at any given point $\mathbf{x}\in C$, (ii) computing its gradient $\nabla F(\mathbf{x})$ at any given point $\mathbf{x}\in C$, and (iii) maximizing a linear function over $C$, in which case, $C$ is called a solvable convex set.\qed
\end{enumerate}
\end{assumption}


Let $\mathbf{x}^*$ be an optimal solution of (\ref{eq:sub_max_prob})
\[
\mathbf{x}^*\in\operatorname*{argmax}\limits_{\mathbf{x}\in C} F(\mathbf{x}).
\]

Since $F$ is DR-submodular, we have the following well-known inequality~\citep{feldman2011unified,hassani2017gradient}: $\forall \mathbf{x}, \mathbf{y}\in [0,1]^n$:
\begin{eqnarray}\label{eq:DR-sub-LB}
\langle \nabla F(\mathbf{x}), \mathbf{y}-\mathbf{x}\rangle  &\ge& F(\mathbf{x}\vee \mathbf{y})+F(\mathbf{x}\wedge \mathbf{y})-2F(\mathbf{x}).
\end{eqnarray}

Different upper bounds $U_t$ on $F(\mathbf{x}^*)$ will be derived from the last inequality (\ref{eq:DR-sub-LB}), depending on the objective function $F$ and the feasible region $C$ of the problem. 
For illustration purpose, in the following examples, we closely follow the four-step process as prescribed earlier.

\subsection{Example 1: Nonnegative monotone DR-submodular and solvable convex set}\label{subsec:mono-exp1}

Assume that $F$ is a non-negative monotone differentiable DR-submodular function, and $C$ is a solvable convex set. 
\begin{description}
\item [Step 1.] The Lyapunov function is 
\[
E(\mathbf{x}(t))=a_tF(\mathbf{x}(t))-b_tF(\mathbf{x}^*).
\]
Its derivative is 
\[
\dot{E}(\mathbf{x}(t))=\dot{a}_t F(\mathbf{x}(t))+a_t\langle\nabla F(\mathbf{x}(t), \dot{\mathbf{x}}(t) \rangle -\dot{b}_tF(\mathbf{x}^*).
\]
\item [Step 2.] The upper bound $U_t$ follows from (\ref{eq:DR-sub-LB}) by taking $\mathbf{x}=\mathbf{x}(t)$, $\mathbf{y}=\mathbf{x}(t)\vee \mathbf{x}^*$, and applying monotonicity: 
\begin{eqnarray*}
\langle \nabla F(\mathbf{x}(t)), \mathbf{x}^*\rangle&\overset{F\text{ monotone}}{\ge}& \langle \nabla F(\mathbf{x}(t)), \underbrace{\mathbf{x}(t)\vee \mathbf{x}^*-\mathbf{x}(t)}_{\le \mathbf{x}^*}\rangle\\
&\ge& F(\mathbf{x}(t)\vee \mathbf{x}^*)-F(\mathbf{x}(t))\overset{F\text{ monotone}}{\ge} F(\mathbf{x}^*)-F(\mathbf{x}(t)),
\end{eqnarray*}
implying that 
\begin{eqnarray}
\nonumber F(\mathbf{x}^*)&\le& F(\mathbf{x}(t))+\langle \nabla F(\mathbf{x}(t)), \mathbf{x}^*\rangle\\
\nonumber &\le& \max_{\mathbf{v}\in C}\left[F(\mathbf{x}(t))+\langle \nabla F(\mathbf{x}(t)), \mathbf{v}\rangle\right]\\
\label{eq:opt-ub-mono}&=& U_t:= F(\mathbf{x}(t))+\langle \nabla F(\mathbf{x}(t)), \mathbf{v}(\mathbf{x}(t))\rangle, 
\end{eqnarray}
where
\begin{equation}\label{eq:v-mono-dr-sub-gene-C}
\mathbf{v}(\mathbf{x}(t)):=\arg\max_{\mathbf{v}\in C}\left[F(\mathbf{x}(t))+\langle \nabla F(\mathbf{x}(t)), \mathbf{v}\rangle\right]=\arg\max_{\mathbf{v}\in C}\langle \nabla F(\mathbf{x}(t)), \mathbf{v}\rangle
\end{equation}
This is an optimization problem with linear objective function over a solvable convex set $C$. An efficient oracle exists for this problem by Assumption~\ref{assum:sub-oracle}(3-iii). Note that in this example $c(\mathbf{x}(t))=1$, $d(t)=1$, and $u(\mathbf{x}(t))=\mathbf{v}(\mathbf{x}(t))$. There is no extra constraint on $O$. 

Because (\ref{eq:opt-ub-mono}) is our desired $U_t$ which involves no quantity surrounding $\mathbf{x}^*$, we obtain
\begin{eqnarray*}
\dot{E}(\mathbf{x}(t))&\ge &\left(\dot{a}_t-\dot{b}_t\right)F(\mathbf{x}(t))+\left\langle\nabla F(\mathbf{x}(t)), a_t\dot{\mathbf{x}}(t)-\dot{b}_t\mathbf{v}(\mathbf{x}(t))\right\rangle
\end{eqnarray*}
\item [Step 3.] A sufficient condition to guarantee that $\dot{E}(\mathbf{x}(t))\ge 0$ is as follows:
\begin{eqnarray*}
\dot{a}_t&=& \dot{b}_t\\
\nonumber\dot{\mathbf{x}}(t)&=& \frac{\dot{b}_t}{a_t}\mathbf{v}(\mathbf{x}(t))=\frac{\dot{a}_t}{a_t}\mathbf{v}(\mathbf{x}(t)),
\end{eqnarray*}
where $\mathbf{v}(\mathbf{x}(t))$ is defined in (\ref{eq:v-mono-dr-sub-gene-C}).

Based on these choices, our algorithm is presented below. 
\begin{algorithm}[H]
    \caption{}
    \label{alg:FW-Con-time-Mono}
    \begin{algorithmic}[1] 
        \State Input: a non-negative monotone differentiable DR-submodular function $F$; a solvable convex set $C$; and $(a, b)\in Q$ satisfying:
        \begin{equation}
        \label{eq:ode_alg_a-t-b-t-mono}\dot{a}_t=\dot{b}_t
        \end{equation}
        \Procedure{Con-Time-Mono}{$F,C, a, b$} 
            \State $\mathbf{x}(0)=\mathbf{0}\in C$ \Comment{Initialization}
            \For{$t\in [0, T]$} \Comment{$T$ is the stopping time}
             \begin{eqnarray*}
\mathbf{v}(\mathbf{x}(t))&=&\arg\max_{\mathbf{v}\in C}\langle\nabla F(\mathbf{x}(t)), \mathbf{v}\rangle\\
\dot{\mathbf{x}}(t)&=&\frac{\dot{a}_t}{a_t}\mathbf{v}(\mathbf{x}(t))=\left(\ln a_t\right)^{\prime}\mathbf{v}(\mathbf{x}(t))
\end{eqnarray*}
\EndFor\label{FW-Confor-time-mono}
            \State \textbf{return} $\mathbf{x}(T)$ 
        \EndProcedure
    \end{algorithmic}
\end{algorithm}

This problem was studied in \cite{calinescu2011maximizing}) and is a variant of the classical Frank-Wolfe algorithm~\citep{frank1956algorithm}). We can recover their \emph{continuous greedy} algorithm by choosing $a_t=e^t, t\in [0,1]$. 

With an algorithm designed under the guidance of the Lyapunov function, we need to make sure the final output is a feasible solution. 
\[
\mathbf{x}(T)=\int_0^T\left(\ln a_t\right)^{\prime}\mathbf{v}(\mathbf{x}(t))dt=\int_0^T\mathbf{v}(\mathbf{x}(t))d\ln a_t
\]
Since $\mathbf{v}(\mathbf{x}(t))\in C, \forall t\in [0, T]$, a sufficient condition for $\mathbf{x}(T)\in C$ is that we choose $a_t$ such that $\ln a_t$ is a cumulative distribution function (cdf) on $[0, T]$. Then $\mathbf{x}(T)$ is a convex combination of feasible solutions and hence also feasible. So the extra constraint set $P$ demands that $\ln a_t$ is a cdf or equivalently $\ln a_0=0, \ln a_T=1, \left(\ln a_t\right)^{\prime}\ge 0$ and a sufficient condition is 
\begin{equation}\label{eq:P-example1}
P=\{(a, b, T): (a, b)\in Q, \ln a_0=0, \ln a_T=1, \dot{a}_t>0\},
\end{equation}
where $Q$ is defined in Assumption~\ref{assume:q-b}.  

\item [Step 4.] Finally, we choose $a_t$, $b_t$ and $T$ to maximize the approximation ratio in (\ref{eq:appro_ratio}) subject to (\ref{eq:ode_alg_a-t-b-t-mono}), $(a, b)\in Q$, 
$(a, b, T)\in P$, and $\ln a_t$ is a cdf on $[0, T]$, where $Q$ and $P$ are defined in Assumption~\ref{assume:q-b} and (\ref{eq:P-example1}), respectively. 
\begin{eqnarray*}
&\sup\limits_{a_t, b_t,T}&\frac{b_T-b_0}{a_T}\\
&\text{s.t.} & \dot{a}_t=\dot{b}_t\\
&& a_0>0, \dot{a}_t\ge 0; b_0\ge 0, \dot{b}_t\ge 0\\
&&\ln a_0=0, \ln a_T=1, \dot{a}_t>0
\end{eqnarray*}
Or equivalently by removing $b_t$,
\begin{eqnarray*}
 &\sup\limits_{a_t, T}&\frac{a_T-a_0}{a_T}\\
 &\text{s.t.}& a_0>0, \dot{a}_t\ge 0;\\
 & & \ln a_0=0, \ln a_T=1, \dot{a}_t>0
\end{eqnarray*}

We now show that the optimal approximation ratio is $1-\frac{1}{e}$ for Algorithm~\ref{alg:FW-Con-time-Mono}. 
\begin{proof}
First, the last constraint implies that the objective value is equal to 
\[
\frac{b_T-b_0}{a_T}=1-\frac{1}{a_T} = 1-\frac{1}{e}.
\]

Second, choose $a_t=b_t=e^t$ and $T=1$. Then $\ln a_t=t$ is a cdf and the approximation ratio is 
\[
F(\mathbf{x}(1))\ge \frac{b_1-b_0}{a_1}F(\mathbf{x}^*)=\left(1-\frac{1}{e}\right)F(\mathbf{x}^*),
\]
because $F(\mathbf{x}(0))\ge 0$, as in Assumption~\ref{assum:sub-oracle}(2). 

Together, the best approximation ratio of the algorithm is $1-\frac{1}{e}$.
\end{proof}

\end{description}

Reading these four-step backward is a rigorous proof that Algorithm~\ref{alg:FW-Con-time-Mono} has an approximation ratio $1-\frac{1}{e}$. However, a shorter and nicer proof can be constructed following the conventional theorem-proof style. For this problem only, below we present such a proof which can be safely skipped without affecting the flow of the remaining article. This proof is a case-in-point on how previous literature has used Lyapunov function only as a {proof technique}, while usually being silent on explaining the {design} of algorithms.

\begin{shaded}
Define the following Lyapunov function associated with Algorithm~\ref{alg:FW-Con-time-Mono} as follows:
\begin{eqnarray}
\label{eq:fw_continuous_Lyapunov_function_monotonoe_sub}
E(\mathbf{x}(t)) &=&e^t(F(\mathbf{x}(t))-F(\mathbf{x}^*))
\end{eqnarray}
\begin{lem}
$E(\mathbf{x}(t))$ is a non-decreasing function.
\end{lem}
\begin{proof}
\begin{eqnarray*}
\dot{E}(\mathbf{x}(t)) &=&e^t\left(F(\mathbf{x}(t))+\langle\nabla f(\mathbf{x}(t)), \dot{\mathbf{x}}(t)\rangle\right)-e^tF(\mathbf{x}^*)\\
&\overset{(\ref{eq:opt-ub-mono})}{\ge}&e^t\left(F(\mathbf{x}(t))+\langle\nabla f(\mathbf{x}(t)), \dot{\mathbf{x}}(t)\rangle\right)-e^t\left( F(\mathbf{x}(t))+\langle\nabla f(\mathbf{x}(t)), \dot{\mathbf{x}}(t)\rangle\right)\\
&=&0
\end{eqnarray*}
\end{proof}
\begin{thm} 
The approximation ratio of Algorithm~\ref{alg:FW-Con-time-Mono} is $1-\frac{1}{e}$.
\end{thm}
\begin{proof}
The approximation ratio is implied by the monotonicity of $E(\mathbf{x}(t))$:
\begin{eqnarray*}
E(1)\ge E(0)&\iff& e(F(\mathbf{x}(1))-F(\mathbf{x}^*)) \ge F(\mathbf{x}(0))-F(\mathbf{x}^*) \\
&\iff& F(\mathbf{x}(1))\ge (1-e^{-1})F(\mathbf{x}^*)
\end{eqnarray*}
assuming that $F(0)\ge 0$.
\end{proof}
\end{shaded}
Although this conventional theorem-proof style takes much less time and easier to read than the one we presented in Section~\ref{subsec:mono-exp1}, however, the extra time spent on wondering how these nice proofs are constructed in the first place can be taxing. Hopefully our framework has assisted in the latter aspect. 

\subsection{Example 2: Nonnegative non-monotone DR-submodular and down-closed solvable convex set}\label{subsec:mono-exp2}

Assume that $F$ is a non-negative non-monotone differentiable DR-submodular function, and $C$ is a solvable down-closed convex set, where $C$ is down-closed if $0\le \mathbf{x}\le \mathbf{y}\in C$ then $\mathbf{x}\in C$.

\begin{description}
\item [Step 1.] The Lyapunov function is 
\[
E(\mathbf{x}(t))=a_tF(\mathbf{x}(t))-b_tF(\mathbf{x}^*)
\]
Its derivative is 
\[
\dot{E}(\mathbf{x}(t))=\dot{a}_t F(\mathbf{x}(t))+a_t\langle\nabla F(\mathbf{x}(t), \dot{\mathbf{x}}(t) \rangle -\dot{b}_tF(\mathbf{x}^*)
\]

\item [Step 2.] The upper bound $U_t$ follows from (\ref{eq:DR-sub-LB}) by taking $\mathbf{x}=\mathbf{x}(t)$ and $\mathbf{y}:=\mathbf{x}(t)\vee \mathbf{x}^*$,
\[
\langle \nabla F(\mathbf{x}(t)), \mathbf{x}(t)\vee \mathbf{x}^*-\mathbf{x}(t)\rangle  \ge F(\mathbf{x}(t)\vee \mathbf{x}^*)-F(\mathbf{x}(t))
\ge (1-\theta(\mathbf{x}(t)))F(\mathbf{x}^*)-F(\mathbf{x}(t))
\]
where $1>\theta(\mathbf{x}(t))\ge \|\mathbf{x}(t)\|_{\infty}=\max_{i=1}^n x_i(t)$ and the last inequality is from~\cite[Lemma 3.5]{feldman2011unified}. Therefore
\begin{eqnarray}
\nonumber F(\mathbf{x}^*)&\le& \frac{F(\mathbf{x}(t))+\langle \nabla F(\mathbf{x}(t)), \mathbf{x}(t)\vee \mathbf{x}^* -\mathbf{x}(t)\rangle }{1-\theta(\mathbf{x}(t))}\\
\label{eq:opt-ub-down-closed-max-prob1} &\overset{\mathbf{x}^*\in C}{\le}& \max_{\mathbf{y}\in C}\frac{F(\mathbf{x}(t))+\langle \nabla F(\mathbf{x}(t)), \mathbf{x}(t)\vee \mathbf{y} -\mathbf{x}(t)\rangle }{1-\theta(\mathbf{x}(t))}\\
\label{eq:opt-ub-down-closed-max-prob1-equ}&=&\max_{\mathbf{v}=\mathbf{x}(t)\vee \mathbf{y} -\mathbf{x}(t), \mathbf{y}\in C}\frac{F(\mathbf{x}(t))+\langle \nabla F(\mathbf{x}(t)), v\rangle }{1-\theta(\mathbf{x}(t))}\\
\label{eq:opt-ub-down-closed-max-prob2}&\le& \max_{\mathbf{v}\in C, \mathbf{v}\le \mathbf{1}-\mathbf{x}}\frac{F(\mathbf{x}(t))+\langle \nabla F(\mathbf{x}(t)), v\rangle }{1-\theta(\mathbf{x}(t))}\\
\label{eq:opt-ub-down-closed}&=&\frac{F(\mathbf{x}(t))+\langle \nabla F(\mathbf{x}(t)),  \mathbf{v}(\mathbf{x}(t)) \rangle }{1-\theta(\mathbf{x}(t))},
\end{eqnarray}
where
\begin{equation}\label{eq:v-opt-donw-closed}
\mathbf{v}(\mathbf{x}(t)):=\arg\max_{\mathbf{v}\in C, \mathbf{v}\le \mathbf{1}-\mathbf{x}}\left[\frac{F(\mathbf{x}(t))+\langle \nabla F(\mathbf{x}(t)), v\rangle }{1-\theta(\mathbf{x}(t))} \right]=\arg\max_{\mathbf{v}\in C, \mathbf{v}\le \mathbf{1}-\mathbf{x}}\langle \nabla F(\mathbf{x}(t)), \mathbf{v}\rangle
\end{equation}

Note that in this example, $c(\mathbf{x}(t))=\frac{1}{1-\theta(\mathbf{x}(t))}$, $d(t)=\frac{1}{1-\theta(\mathbf{x}(t))}$ and $\mathbf{u}(\mathbf{x}(t))=\mathbf{v}(\mathbf{x}(t))$. We have the following extra constraints collected in 
\begin{equation}\label{eq:O-down-closed}
O:=\left\{c(\mathbf{x}(t)): 0<\frac{1}{c}\le 1-\|x(t)\|_{\infty}, t\in [0, T]\right\}.
\end{equation}

Note further that the maximization problem in (\ref{eq:opt-ub-down-closed-max-prob1}) is NP-hard in general because its objective function is piece-wise linear convex. To overcome this difficulty, we continue relaxing the upper bound to obtain a new maximization problem in (\ref{eq:opt-ub-down-closed-max-prob2}), which has a linear objective function in the new decision variable $v$ and a new solvable region. An efficient oracle exists for this problem by Assumption~\ref{assum:sub-oracle}(3-iii). 

We now show that (\ref{eq:opt-ub-down-closed-max-prob2}) is indeed a relaxation of (\ref{eq:opt-ub-down-closed-max-prob1-equ}), which is equivalent to (\ref{eq:opt-ub-down-closed-max-prob1}). 
\begin{proof}
First they have the same objective function. Second we show that the down-closeness of $C$ implies that (\ref{eq:opt-ub-down-closed-max-prob2}) has a larger feasible region than that of (\ref{eq:opt-ub-down-closed-max-prob1-equ}). $\forall \mathbf{v}, \mathbf{y}\in \{(\mathbf{v},\mathbf{y}): \mathbf{v}=\mathbf{x}(t)\vee \mathbf{y} -\mathbf{x}(t), \mathbf{y}\in C\}$, we have
\begin{eqnarray*}
\mathbf{v}&=&\mathbf{x}(t)\vee \mathbf{y} -\mathbf{x}(t)\le \mathbf{y}\in C\Longrightarrow \mathbf{v}\in C\\
\mathbf{v}&=&\mathbf{x}(t)\vee \mathbf{y} -\mathbf{x}(t)\overset{\mathbf{y}\le \mathbf{1}}\le \mathbf{1}-\mathbf{x}(t)
\end{eqnarray*}
Namely,
\[
\mathbf{v}\in \{\mathbf{v}\in C: \mathbf{v}\le \mathbf{1}-\mathbf{x}\}
\]
\end{proof}

Because (\ref{eq:opt-ub-down-closed}) is our desired $U_t$ which involves no quantity surrounding $\mathbf{x}^*$, we obtain 
\begin{eqnarray}
\label{eq:L-fun-derivative-down-closed}\dot{E}(\mathbf{x}(t))&\ge &\left(\dot{a}_t-\frac{\dot{b}_t}{1-\theta(\mathbf{x}(t))}\right)F(\mathbf{x}(t))+\left\langle\nabla F(\mathbf{x}(t)), a_t\dot{\mathbf{x}}(t)-\frac{\dot{b}_t\mathbf{v}(\mathbf{x}(t))}{1-\theta(\mathbf{x}(t))}\right\rangle
\end{eqnarray}
\item [Step 3.] Guided by the last quantity (\ref{eq:L-fun-derivative-down-closed}) and the condition (\ref{eq:O-down-closed}) imposed in Step 2., we are ready to design an algorithm to guarantee that $\dot{E}(t)\ge 0$. A sufficient condition is the following:
\begin{eqnarray}
\nonumber\dot{a}_t&\ge & \frac{\dot{b}_t}{1-\theta(\mathbf{x}(t))}\\
\label{eq:x-down-closed}\dot{\mathbf{x}}(t)&=& \frac{\dot{b}_t}{a_t}\frac{\mathbf{v}(\mathbf{x}(t))}{1-\theta(\mathbf{x}(t))}\\
\label{eq:theta-down-closed}1&>&\theta(\mathbf{x}(t))\ge \|x(t)\|_{\infty},
\end{eqnarray}
where $\mathbf{v}(\mathbf{x}(t))$ is defined in (\ref{eq:v-opt-donw-closed}). We show that a sufficient condition for the above to hold is as follows:
\begin{eqnarray}
\label{eq:theta-down-closed-equi}\theta(\mathbf{x}(t))&=&1-{\frac{a_0}{a_t}}=1-{\frac{\dot{b}_t}{\dot{a}_t}}\\
\nonumber \dot{a}_t&=& \frac{\dot{b}_t}{1-\theta(\mathbf{x}(t))}=\frac{\dot{b}_t}{{\frac{a_0}{a_t}}}\\
\label{eq:x-down-closed-equi}\dot{\mathbf{x}}(t)&=& \frac{\dot{b}_t}{a_t}\frac{\mathbf{v}(\mathbf{x}(t))}{1-\theta(\mathbf{x}(t))}=\frac{\dot{a}_t}{a_t}\mathbf{v}(\mathbf{x}(t)).
\end{eqnarray}
\begin{proof}
From the ode (\ref{eq:x-down-closed-equi}), we have, $\forall i=1,\ldots, n$, 
\[
\dot{x}_i(t)=\frac{\dot{a}_t}{a_t}\mathbf{v}_i(\mathbf{x}(t))\overset{\mathbf{v}_i(\mathbf{x}(t))\le 1-x_i(t)}{\le} \frac{\dot{a}_t}{a_t}(1-x_i(t)), 
\]
where the inequality is from (\ref{eq:v-opt-donw-closed}). From the Gr\"{o}nwall's inequality~\citep{gronwall1919note}, the last ordinary differential inequality is equivalent to 
\begin{equation*}
1-x_i(t)\ge e^{-\int_0^t \frac{\dot{a}_s}{a_s}ds}=e^{-\ln \frac{a_t}{a_0}} =\frac{a_0}{a_t}>0,
\end{equation*}
implying that
\begin{equation*}
 1-\|x(t)\|_{\infty} \ge \frac{a_0}{a_t}>0.
\end{equation*}
Therefore (\ref{eq:theta-down-closed-equi}) implies (\ref{eq:theta-down-closed}).
\end{proof}

Based on these choices, our algorithm is presented below. 
\begin{algorithm}[H]
    \caption{}
    \label{alg:FW-Con-time-down-closed}
    \begin{algorithmic}[1] 
        \State Input: a non-negative non-monotone differentiable DR-submodular function $F$; a solvable down-closed convex set $C$; and $(a, b)\in Q$ satisfying:
        \begin{equation}
        \label{eq:ode_alg_a-t-b-t-down-closed}\dot{a}_t=\frac{\dot{b}_t}{\frac{a_0}{a_t}}\iff \dot{b}_t=a_0\left(\ln a_t\right)^{\prime}\iff b_t-b_0=a_0\ln\frac{a_t}{a_0}
        \end{equation}
        \Procedure{Con-Time-Down-Closed}{$F,C, a, b$} 
            \State $\mathbf{x}(0)=\mathbf{0}\in C$ \Comment{Initialization}
            \For{$t\in [0, T]$} \Comment{$T$ is the stopping time}
             \begin{eqnarray}
\label{eq:ode_alg_e-t-e-v-down-closed} 
 \mathbf{v}(\mathbf{x}(t))&=&\arg\max_{\mathbf{v}\in C, \mathbf{v}\le \mathbf{1}-\mathbf{x}}\langle\nabla F(\mathbf{x}(t)),\mathbf{v}\rangle\\
\label{eq:ode_alg_e-t-e-t-down-closed}
\dot{\mathbf{x}}(t)&=&\frac{\dot{a}_t}{a_t}\mathbf{v}(\mathbf{x}(t))=\left(\ln a_t\right)^{\prime}\mathbf{v}(\mathbf{x}(t))
\end{eqnarray}
\EndFor\label{FW-Confor-time-down-closed}
            \State \textbf{return} $\mathbf{x}(T)$ 
        \EndProcedure
    \end{algorithmic}
\end{algorithm}
This problem was studied in~\cite{feldman2011unified,bian2017continuous}. We can recover their \emph{measured continuous greedy} algorithm by choosing $a_t=e^t, b_t=t, t\in [0, 1]$.

With an algorithm designed under the guidance of the Lyapunov function, we now need to address the feasibility of the final output. From (\ref{eq:ode_alg_e-t-e-t-down-closed}), we obtain that
\[
\mathbf{x}(T)=\int_0^T \mathbf{v}(\mathbf{x}(t))d(\ln a_t)
\]
Note that (\ref{eq:ode_alg_e-t-e-v-down-closed}) implies that $\mathbf{v}(\mathbf{x}(t))\in C, \forall t\in [0, T]$. If we choose $a_t$ such that $\ln a_t$ is a cdf on $[0, T]$, then the right-hand side of last inequality is a feasible solution because it is a convex combination of feasible solutions. So the extra constraint set $P$ demands that $\ln a_t$ is a cdf or equivalently 
\begin{equation}\label{eq:P-example2}
P=\{(a, T): a\in Q, \ln a_0=0, \ln a_T=1, \dot{a}_t>0\},
\end{equation}
where $Q$ is defined in Assumption~\ref{assume:q-b}.

\item [Step 4.] Finally, the approximation ratio of this algorithm is to choose $a_t$, $b_t$ and $T$ to  maximize the approximation ratio in (\ref{eq:appro_ratio}) subject to (\ref{eq:ode_alg_a-t-b-t-down-closed}), $(a, b)\in Q$, and $(a, b, T)\in P$, where $Q$ and $P$ are defined in Assumption~\ref{assume:q-b} and (\ref{eq:P-example2}), respectively.
\begin{eqnarray}
\nonumber &\sup\limits_{a_t, b_t, T}& \frac{b_T-b_0}{a_T}\\
\label{eq:a-b-temp}&\text{s.t.}&b_t-b_0=a_0\ln\frac{a_t}{a_0}\\
&\nonumber & a_0>0, \dot{a}_t>0; b_0\ge 0, \dot{b}_t>0\\
\nonumber&& \ln a_0=0, \ln a_T=1, \dot{a}_t>0
\end{eqnarray}
Or equivalently by removing $b_t$,
\begin{eqnarray*}
\nonumber &\sup\limits_{a_t, T}&\frac{\ln\frac{a_T}{a_0}}{\frac{a_T}{a_0}}\\
\nonumber &\text{s.t.} &a_0>0, \dot{a}_t>0;\\
\nonumber&&\ln a_0=0, \ln a_T=1, \dot{a}_t>0
\end{eqnarray*}

We now show that the optimal approximation ratio is $\frac{1}{e}$ for Algorithm~\ref{alg:FW-Con-time-down-closed}. 
\begin{proof}
First, the last constraint implies that the objective value is equal to
\[
\frac{\ln\frac{a_T}{a_0}}{\frac{a_T}{a_0}}= \frac{1}{e}
\]

Second, choose $a_t=e^t, b_t=\ln a_t=t$ and $T=1$. Then $\ln a_t=t$ is a cdf and the approximation ratio is 
\[
F(\mathbf{x}(1))\ge \frac{b_1-b_0}{a_1}F(\mathbf{x}^*)=\frac{1}{e}F(\mathbf{x}^*)
\]
because $F(\mathbf{x}(0))\ge 0$, as in Assumption~\ref{assum:sub-oracle}(2).

Together, the best approximation ratio of the algorithm is $1/e$.
\end{proof}

\end{description}

\subsection{Exampel 3: Nonnegative non-monotone DR-submodular and solvable convex set}\label{subsec:con-exp3}

\begin{description}
\item [Step 1.] The Lyapunov function is 
\[
E(\mathbf{x}(t))=a_tF(\mathbf{x}(t))-b_tF(\mathbf{x}^*)
\]
Its derivative is 
\[
\dot{E}(\mathbf{x}(t))=\dot{a}_t F(\mathbf{x}(t))+a_t\langle\nabla F(\mathbf{x}(t), \dot{\mathbf{x}}(t) \rangle -\dot{b}_tF(\mathbf{x}^*)
\]

\item [Step 2.] The upper bound $U_t$ follows from (\ref{eq:DR-sub-LB}) by taking $\mathbf{x}=\mathbf{x}(t)$ and $\mathbf{y}:=\mathbf{x}^*$,
\begin{eqnarray*}
\langle \nabla F(\mathbf{x}(t)), \mathbf{x}^*-\mathbf{x}(t)\rangle  &\ge& F(\mathbf{x}(t)\vee \mathbf{x}^*)+F(\mathbf{x}(t)\wedge \mathbf{x}^*)-2F(\mathbf{x}(t))\\
&\ge& (1-\theta(\mathbf{x}(t)))F(\mathbf{x}^*)-2F(\mathbf{x}(t)),
\end{eqnarray*}
where $1>\theta(\mathbf{x}(t))\ge \|\mathbf{x}(t)\|_{\infty}=\max_{i=1}^n x_i(t)$ and the last inequality is from~\cite[Lemma 3.5]{feldman2011unified}. Therefore
\begin{eqnarray}
\nonumber F(\mathbf{x}^*)&\le& \frac{2F(\mathbf{x}(t))+\langle \nabla F(\mathbf{x}(t)), \mathbf{x}^*-\mathbf{x}(t)\rangle }{1-\theta(\mathbf{x}(t))}\\
\nonumber &\le& \max_{\mathbf{v}\in C}\frac{2F(\mathbf{x}(t))+\langle \nabla F(\mathbf{x}(t)), v-\mathbf{x}(t)\rangle }{1-\theta(\mathbf{x}(t))}\\
\label{eq:opt-ub-general}&=&\frac{2F(\mathbf{x}(t))+\langle \nabla F(\mathbf{x}(t)), \mathbf{v}( \mathbf{x}(t)-\mathbf{x}(t)\rangle }{1-\theta(\mathbf{x}(t))},
\end{eqnarray}
where
\begin{equation}\label{eq:v-exp3}
\mathbf{v}(\mathbf{x}(t)):=\arg\max_{\mathbf{v}\in C}\left[\frac{2F(\mathbf{x}(t))+\langle \nabla F(\mathbf{x}(t)), v-\mathbf{x}(t)\rangle }{1-\theta(\mathbf{x}(t))} \right]=\arg\max_{\mathbf{v}\in C}\langle \nabla F(\mathbf{x}(t)), v\rangle
\end{equation}

Note that in this example, $c(\mathbf{x}(t))=\frac{2}{1-\theta(\mathbf{x}(t))}$, $d(t)=\frac{1}{1-\theta(\mathbf{x}(t))}$ and $\mathbf{u}(\mathbf{x}(t))=\mathbf{v}(\mathbf{x}(t))-\mathbf{x}(t)$. We have the following extra constraints collected in 
\begin{equation}\label{eq:O-general}
O:=\left\{c(\mathbf{x}(t)): 0<\frac{2}{c}\le 1-\|x(t)\|_{\infty}, t\in [0, T]\right\}.
\end{equation}

Note further the last optimization has a linear objective function. An efficient oracle exists for this problem by Assumption~\ref{assum:sub-oracle}(3-iii).

Because (\ref{eq:opt-ub-general}) is our desired $U_t$ which involves no quantity surrounding $\mathbf{x}^*$, we obtain 
\begin{eqnarray}
\label{eq:L-fun-derivative-exp3}\dot{E}(t)&=&\left(\dot{a}_t-\frac{2\dot{b}_t}{1-\theta(\mathbf{x}(t))}\right)F(\mathbf{x}(t))+\left\langle\nabla F(\mathbf{x}(t)), a_t\dot{\mathbf{x}}(t)-\frac{\dot{b}_t(\mathbf{v}(\mathbf{x}(t))-\mathbf{x}(t))}{1-\theta(\mathbf{x}(t))}\right\rangle
\end{eqnarray}

\item [Step 3.] Guided by the last quantity (\ref{eq:L-fun-derivative-exp3}) and the condition (\ref{eq:O-general}) imposed in Step 2, we are ready to design an algorithm to guarantee that $\dot{E}(t)\ge 0$. A sufficient condition is as follows:
\begin{eqnarray}
\nonumber\dot{a}_t&\ge& \frac{2\dot{b}_t}{1-\theta(\mathbf{x}(t))}\\
\label{eq:x-general}\dot{\mathbf{x}}(t)&=& \frac{\dot{b}_t}{a_t}\frac{\mathbf{v}(\mathbf{x}(t))-\mathbf{x}(t)}{1-\theta(\mathbf{x}(t))}\\
\label{eq:theta-general}1&>&\theta(\mathbf{x}(t))\ge\|x(t)\|_{\infty},
\end{eqnarray}
where $\mathbf{v}(\mathbf{x}(t))$ is defined in (\ref{eq:v-exp3}). We show that a sufficient condition for the above to hold is as follows:
\begin{eqnarray}
\label{eq:theta-general-equi}\theta(\mathbf{x}(t))&=&1-\sqrt{\frac{a_0}{a_t}}\\
\nonumber\dot{a}_t&=& \frac{2\dot{b}_t}{1-\theta(\mathbf{x}(t))}=\frac{2\dot{b}_t}{\sqrt{\frac{a_0}{a_t}}}\\
\label{eq:x-general-equi}\dot{\mathbf{x}}(t)&=& \frac{\dot{b}_t}{a_t}\frac{\mathbf{v}(\mathbf{x}(t))-\mathbf{x}(t)}{1-\theta(\mathbf{x}(t))}=\frac{\dot{a}_t}{2a_t}(\mathbf{v}(\mathbf{x}(t))-\mathbf{x}(t))
\end{eqnarray}
\begin{proof}
From the ode (\ref{eq:x-general-equi}), we have, $\forall i=1,\ldots, n$, 
\[
\dot{x}_i(t)=\frac{\dot{a}_t}{2a_t}(\mathbf{v}(\mathbf{x}(t))-\mathbf{x}(t))\overset{\mathbf{v}\in C}{\le} \frac{\dot{a}_t}{2a_t}(1-x_i(t)),
\]
where the inequality is from (\ref{eq:v-exp3}). From the Gr\"{o}nwall's inequality~\citep{gronwall1919note}, the last ordinary differential inequality is equivalent to
\begin{equation*}
1-x_i(t)\ge e^{-\int_0^t \frac{\dot{a}_s}{2a_s}ds}=e^{-\ln \sqrt{\frac{a_t}{a_0}} }=\sqrt{\frac{a_0}{a_t}}>0, 
\end{equation*}
implying that
\[
 1-\|x(t)\|_{\infty}\ge \sqrt{\frac{a_0}{a_t}}>0.
\]
Therefore (\ref{eq:theta-general-equi}) implies (\ref{eq:theta-general}).
\end{proof}

Based on these choices, our algorithm is presented below. 
\begin{algorithm}[H]
    \caption{}
    \label{alg:FW-Con-time}
    \begin{algorithmic}[1] 
        \State Input: a non-negative non-monotone differentiable DR-submodular function $F$; a solvable convex set $C$; and $(a, b)\in Q$ satisfying:
        \begin{equation}
        \label{eq:ode_alg_a-t-b-t-general}\dot{a}_t=\frac{2\dot{b}_t}{\sqrt{\frac{a_0}{a_t}}}\iff \dot{b}_t=\sqrt{a_0}\left(\sqrt{a_t}\right)^{\prime}\iff b_t-b_0=\sqrt{a_0}(\sqrt{a_t}-\sqrt{a_0})
        \end{equation}
        \Procedure{FW-Con-Time}{$F,C, a, b$} 
            \State $\mathbf{x}(0)=\mathbf{0}\in C$ \Comment{Initialization}
            \For{$t\in [0, T]$} \Comment{$T$ is the stopping time}
             \begin{eqnarray}
\nonumber \mathbf{v}(\mathbf{x}(t))&=&\arg\max_{\mathbf{v}\in C}\langle\nabla F(\mathbf{x}(t)), \mathbf{v}\rangle\\
\label{eq:ode_alg_e-t-e-t}
\dot{\mathbf{x}}(t)&=&\frac{\dot{a}_t}{2a_t}(\mathbf{v}(\mathbf{x}(t))-\mathbf{x}(t))=\left(\ln \sqrt{a_t}\right)^{\prime}(\mathbf{v}(\mathbf{x}(t))-\mathbf{x}(t)).
\end{eqnarray}
\EndFor\label{FW-Confor-time}
            \State \textbf{return} $\mathbf{x}(T)$ 
        \EndProcedure
    \end{algorithmic}
\end{algorithm}

Note that this is nothing but the continuous-time version of the Frank-Wolfe algorithm. The process we just illustrated shows how Frank-Wolfe type of algorithms arises naturally not only in convex optimization, but also in DR-submodular maximization. (\ref{eq:ode_alg_a-t-b-t-general}) clearly shows that $a_t$ may not equal to $b_t$ in general, in contrast to those Lyapunov function based approaches in convex optimization where $a_t$ always equals to $b_t$ because their main objective is for convergence rate, while ours is for approximation ratio.

This problem was studied in~\cite{durr2021non} and \cite{du2022improved}, with the latter improving the approximation ratio in the former. Algorithm \ref{alg:FW-Con-time} recovers the algorithm in \cite{du2022improved}, while \cite{durr2021non} only considers the discrete-time setting.

With an algorithm designed under the guidance of the Lyapunov function, we now need to address the feasibility of the final output. This was answered affirmatively by \cite{jacimovic1999continuous} who showed that the continuous-time version of the Frank-Wolfe algorithm maintains feasibility at each time $t\in [0, T]$ for any $(a, b)\in Q$. So there is no extra feasibility constraint on $P$.

\item [Step 4.] Finally, the approximation ratio of this algorithm is to choose $a_t$, $b_t$ and $T$ to  maximize the approximation ratio in (\ref{eq:appro_ratio}) subject to (\ref{eq:ode_alg_a-t-b-t-general}) and $(a, b)\in Q$, where $Q$ is defined in Assumption~\ref{assume:q-b}.  
\begin{eqnarray}
\nonumber &\sup\limits_{a_t, b_t, T}& \frac{b_T-b_0}{a_T}\\
\label{eq:a-b-temp}&\text{s.t.}&b_t-b_0=\sqrt{a_0}(\sqrt{a_t}-\sqrt{a_0})\\
&\nonumber & a_0>0, \dot{a}_t>0; b_0\ge 0, \dot{b}_t>0.
\end{eqnarray}
Or equivalently by removing $b_t$,
\begin{eqnarray*}
\nonumber &\sup\limits_{a_t, T}&\sqrt{\frac{a_0}{a_T}}\left(1-\sqrt{\frac{a_0}{a_T}}\right)\\
\nonumber &\text{s.t.} &a_0>0, \dot{a}_t>0.
\end{eqnarray*}

We now show that the optimal approximation ratio is $\frac{1}{4}$ for Algorithm~\ref{alg:FW-Con-time}. 
\begin{proof}
First note that the objective value is bounded above as follows:
\[
\sqrt{\frac{a_0}{a_T}}\left(1-\sqrt{\frac{a_0}{a_T}}\right)\le \frac{1}{4}
\]

Second, there are many different ways to choose $a_t$, $b_t$ and $T$ to achieve 1/4. Below are some examples. Recall Assumption~\ref{assum:sub-oracle}(2), we have $F(\mathbf{x}(0))=F(\mathbf{0})\ge 0$.
\begin{enumerate}[(a)]
\item $a_t=e^t, b_t=e^{t/2}-1$ and $T=2\ln 2\approx 1.386294$. Then the approximation ratio is
\[
F(\mathbf{x}(t))\ge \frac{b_tF(\mathbf{x}^*)+a_0L_0-b_0U_0}{a_1}= \frac{b_t}{a_t}F(\mathbf{x}^*)=\frac{e^{t/2}-1}{e^t}F(\mathbf{x}^*).
\]
The maximum is 1/4, achieved at $t=2\ln 2\approx 1.386294$.

\item $a_t=t+1, b_t=\sqrt{t+1}-1$ and $T=3$. Then the approximation ratio is
\[
F(\mathbf{x}(t))\ge \frac{b_t}{a_t}F(\mathbf{x}^*)=\frac{\sqrt{t+1}-1}{t+1}F(\mathbf{x}^*).
\]
The maximum is 1/4, achieved at $t=3$.

\item $a_t=(t+1)^2, b_t=t$ and $T=1$. Then the approximation ratio is
\[
F(\mathbf{x}(t))\ge \frac{b_t}{a_t}F(\mathbf{x}^*)=\frac{t}{(t+1)^2}F(\mathbf{x}^*).
\]
The maximum is 1/4, achieved at $t=1$.
\end{enumerate}

Together, the best approximation ratio of the algorithm is 1/4.

\end{proof}

Note that if we relax Assumption~\ref{assum:sub-oracle} (1), and initialize the algorithm by an arbitrary feasible solution $\mathbf{x}(0)\in C$, then the approximation ratio becomes
\begin{eqnarray*}
F(\mathbf{x}(T))&\ge & \frac{1}{4}\left(1-\|\mathbf{x}(0)\|_{\infty}\right) F(\mathbf{x}^*).
\end{eqnarray*}
For example, if we choose $\mathbf{x}(0)=\arg\min_{\mathbf{x}\in C}\|\mathbf{x}\|_{\infty}$, then the approximation ratio becomes
\begin{eqnarray*}
F(\mathbf{x}(T))&\ge & \frac{1}{4}\left(1-\min_{\mathbf{x}\in C}\|\mathbf{x}\|_{\infty}\right) F(\mathbf{x}^*).
\end{eqnarray*}

\end{description}

\section{Discretization of the continuous-time algorithm}\label{sec:algo-design-L-function-dis}
In Section~\ref{subsec:algo-design-con-time-poten-fun}, we introduce the concept of potential function which is the discrete-time counterpart of the Lyapunov function. In Section~\ref{subsec:algo-design-dis-time}, we introduce the second phase of the framework by discretizing the continuous-time algorithm to obtain its discrete-time implementation with almost the same approximation ratio along with provable time complexity. In Sections~\ref{sec:sub-max-con-alg-design-via-L-fun-mono-general-examp1}, \ref{sec:sub-max-con-alg-design-via-L-fun-nonmono-down-closed-examp2}, and \ref{subsec:nonmono-examp3}, we illustrate the transformation with almost the same examples as before.

\subsection{Potential function}\label{subsec:algo-design-con-time-poten-fun}
\begin{defi}\label{def:algo-dis} \emph{(Discrete-time algorithm)}
In the discrete-time setting, a (discrete-time) algorithm for problem (\ref{eq:generic_max_prob}) is a sequence of vectors 
\[
\mathbf{x}({t_j})\in \mathbb{R}^n, j=0,1,\ldots, N,
\]
starting from $\mathbf{x}(t_0)\in C$ at times 0 and outputting $\mathbf{x}(t_N)\in C$ at time $t_N$, where $N$ is the total number of iterations during $[0, T]$, and $t_j$ is an increasing function of $j=0, \ldots, N$, indicating the time of $t$ at the start of the $j$-th iteration subject to the following boundary conditions $t_0=0, t_N=T$. \qed
\end{defi}

\begin{defi}\label{def:Lyapunov-fun-dis} \emph{(Discrete-time potential function)} A (discrete-time) potential function $E(\mathbf{x}(t_j)), j=0,\ldots, N$ (the discrete-time counter-part of the Lyapunov function) associated with the algorithm $\mathbf{x}({t_j})$ in Definition \ref{def:algo-dis} has a bounded increment at each step $j$
\[
E(\mathbf{x}(t_{j+1}))-E(\mathbf{x}(t_j))\ge -B_j, j=0, \ldots, N-1,
\]
where $B_j\ge 0$. \qed
\end{defi}

Similar as before, we focus on the following parametric form of the potential function from now on, obtained by sampling the continuous-time Lyapunov function $E(\mathbf{x}(t))$ in (\ref{eq:Lyapunov-function-continuous-time}) at times $t_j, j=0, \ldots, N$:
\begin{equation}\label{eq:potential-fun-dis}
E(\mathbf{x}(t_j))=a_{t_j}F(\mathbf{x}(t_j))-b_{t_j}F(\mathbf{x}^*), j=0,1,\ldots, N.
\end{equation}
Because $E$ has a bounded increment, telescoping over $j=0, \ldots, N-1$ leads to 
\begin{eqnarray*}
-\sum_{j=0}^{N-1}B_j &\le& E(\mathbf{x}(t_N))-E(\mathbf{x}(t_0)) \\
&=&a_{t_N}F(\mathbf{x}(t_N))-a_{t_0}F(\mathbf{x}(t_0))-(b_{t_N}-b_{t_0} )F(\mathbf{x}^*)\\
&\le&a_{t_N}F(\mathbf{x}(t_N))-(b_{t_N}-b_{t_0} )F(\mathbf{x}^*),
\end{eqnarray*}
where the last inequality holds when $F$ is nonnegative, as in Assumption~\ref{assum:sub-oracle}(3). So the approximation ratio is 
\begin{eqnarray}\label{eq:appro_ratio-dis}
F(\mathbf{x}(t_N))&\ge&\frac{b_{t_N}-b_{t_0}}{a_{t_N}}F(\mathbf{x}^*)-\frac{\sum_{j=0}^{N-1}B_j}{a_{t_N}},
\end{eqnarray}

The first term in the last quantity dictates the approximation ratio, while the second term will be used to decide the time complexity of the algorithm.

\subsection{Algorithm design: second phase}\label{subsec:algo-design-dis-time}
In the second phase of our framework, we leverage the first-phase results developed in the continuous-time setting in Section~\ref{subsec:algo-design-con-time} to construct a discrete-time implementation with almost the same approximation ratio along with provable time complexity. 

Before we formally present the second phase, we explain the major considerations and rationals which lead to our design of the second phase.  

One major purpose of the first-phase is to pin down the two functions $a_t$, $b_t$ and the stopping time $T$.  With $a_t$, $b_t$ and $T$ known from the first phase, the discrete-time algorithm samples $N$ times during $[0,T]$:
\[
0=t_0<\ldots<t_N=T,
\]
where each time $t_j$ indicates the time of $t$ at the start of the $j$-th iteration of the algorithm. 

These starting times give us two known increasing sequences $a_{t_j}, b_{t_j}, j=0, \ldots, N$ under Assumption~\ref{assume:q-b}. The corresponding potential function $E(\mathbf{x}(t_j))$ takes the parametric form in (\ref{eq:potential-fun-dis}):
\[
E(\mathbf{x}(t_j))=a_{t_j}F(\mathbf{x}(t_j))-b_{t_j}F(\mathbf{x}^*), j=0,1,\ldots, N.
\]

Due to the discretization errors, we cannot expect that $E(\mathbf{x}(t_j))$ is a non-decreasing function of $j$ as in the continuous-time setting and even belong to $x(t)$ at all. However, Definition~\ref{def:Lyapunov-fun-dis} of the potential function offers us some leeway in $B_j$; namely, if we could show that $E(\mathbf{x}(t_j))$ has a bounded increment at each step, then we can derive an almost identical approximation ratio as in (\ref{eq:appro_ratio-dis}). 

There are two types of discretization errors that may arise from the continuous-time to the discrete-time. The first type comes from the differential operation. While in the continuous-time setting,  we take the first-order derivative of the Lyapunov function to show its monotonically, we need to replace the differential operation with the difference operation in the discrete-time. The second type comes from the discretization of the continuous-time algorithm (\ref{eq:appro-ratio})-(\ref{eq:appro-O}). We now address on how to handle both errors.

To account for the first type of error, we need to establish a connection between function values and their gradients. One standard assumption in studying first-order methods is $L$-smoothness. 
\begin{assumption}\label{assum:L-smmoth}
$F:C\mapsto \mathbb{R}^{+}$ is $L$-smooth: $\forall \mathbf{x}, \mathbf{y}\in C$, 
\[
\|\nabla F(\mathbf{y})-\nabla F(\mathbf{x})\|_2 \leq L\|\mathbf{y}-\mathbf{x}\|_2
\]
\[
\bigg\Downarrow
\]
\begin{eqnarray*}
F(\mathbf{x})+\langle\nabla F(\mathbf{x}), \mathbf{y}-\mathbf{x}\rangle- \frac{L}{2}\|\mathbf{y}-\mathbf{x}\|^{2} \le F(\mathbf{y})\le F(\mathbf{x})+\langle\nabla F(\mathbf{x}), \mathbf{y}-\mathbf{x}\rangle+ \frac{L}{2}\|\mathbf{y}-\mathbf{x}\|^{2}
\end{eqnarray*}\qed
\end{assumption}

The $L$-smoothness above is widely assumed in first-order methods in convex optimization~\citep{nemirovskij1983problem} and in DR-submodular maximization~\citep{bian2020continuous}. For example, the multilinear extension $f^M$ of any submodular set function $\{0,1\}^n\ni x\mapsto f(x) \mathbb{R}_{+}$ is DR-submodular and $L$-smooth~\citep[Lemma 2.3.7.]{feldman2013maximization} with 
\begin{equation}\label{eq:mle-L-smooth}
L=O\left(n^2\right)f(\mathbf{x}^*)=O\left(n^2\right)f^M(\mathbf{x}^*).
\end{equation}

Under the $L$-smoothness assumption and the discrete counterpart of the upper bound $U_{t}$ in (\ref{eq:U_t-form}), the increment of the potential function $E(\mathbf{x}(t_j))$ has the following bound at each iteration $j\in \{0,\ldots, N-1\}$:
\begin{eqnarray}
\nonumber &&E(\mathbf{x}(t_{j+1}))-E(\mathbf{x}(t_j))\\
\nonumber&\ge& (a_{t_{j+1}}-a_{t_j}-c(\mathbf{x}(t_j))(b_{t_{j+1}}-b_{t_j})) F(\mathbf{x}(t_j))\\
\label{eq:alg-design-dis}&&+\langle \nabla F(\mathbf{x}(t_j)), a_{t_{j+1}}(\mathbf{x}(t_{j+1})-\mathbf{x}(t_j))- (b_{t_{j+1}}-b_{t_j}) d(t_j)\mathbf{u}(\mathbf{x}(t_j))\rangle-B_j
\end{eqnarray}
where $c$, $d$, and $\mathbf{u}$ are defined in (\ref{eq:U_t-form}), and 
\begin{eqnarray}
\label{eq:B_j} B_j&:=&a_{t_{j+1}}\frac{L}{2}\|\mathbf{x}(t_{j+1})-\mathbf{x}(t_j)\|^2.
\end{eqnarray} 

For the second type of error, recall the continuous-time algorithm (\ref{eq:appro-ratio})-(\ref{eq:appro-O}), whose discrete counterpart is as follows:
\begin{eqnarray}
\label{eq:appro-ratio-dis-design}a_{t_{j+1}}-a_{t_j}&\ge& c(\mathbf{x}(t_j))(b_{t_{j+1}}-b_{t_j})) , j=0, \ldots, N-1\\
\label{eq:appro-x-dis-design}\mathbf{x}(t_{j+1})-\mathbf{x}(t_j)&=&\frac{b_{t_{j+1}}-b_{t_j}}{a_{t_{j+1}}}d(t_j)\mathbf{u}(\mathbf{x}(t_j)), j=0, \ldots, N-1\\
\label{eq:appro-o-dis-design}c&\in& O,
\end{eqnarray}

However this system may be void if we use the two functions $a_t$ and $b_t$ identified from the continuous-time setting. In (\ref{eq:alg-design-dis}), the sign of the gradient associated with the second equality (\ref{eq:appro-x-dis-design}) of the system above is unknown whenever the objective function is non-monotone. Therefore we normally need to retain this equality, leading to possible violation of the first inequality (\ref{eq:appro-ratio-dis-design}) of the system above and hence introducing the aforementioned second type of error, which can be quantified as $G_j^{+}$, where 
\begin{equation}\label{eq:G-j-dis-design}
G_j=c(\mathbf{x}(t_j))(b_{t_{j+1}}-b_{t_j}))-(a_{t_{j+1}}-a_{t_j}).
\end{equation}

Fortunately, in (\ref{eq:alg-design-dis}), this error is associated with the objective value, whose sign is known if the objective function is non-negative, in which case, we are able to control this error as explained shortly in Step 6. 

Now we are well prepared for presenting a two-step process, which is the second phase of our framework. We continue the numbering from the first phase.
\begin{description}
\item [Step 5. (Discretize the continuous-time algorithm)] With $a_t$, $b_t$ and $T$ known from the first phase, the discrete-time algorithm samples $N$ times during $[0,T]$: $0=t_0<\ldots<t_N=T$, with the following updating rule at the $j$-th iteration: $j=0, \ldots, N-1$,
\begin{eqnarray}
\label{eq:appro-ratio-dis}a_{t_{j+1}}-a_{t_j}&=& c(\mathbf{x}(t_j))(b_{t_{j+1}}-b_{t_j}))-G_j , \\
\label{eq:appro-alg-dis}\mathbf{x}(t_{j+1})-\mathbf{x}(t_j)&=&\frac{b_{t_{j+1}}-b_{t_j}}{a_{t_{j+1}}}d(t_j)\mathbf{u}(\mathbf{x}(t_j)),\\
\label{eq:appro-O-dis} c&\in& O,
\end{eqnarray}
where $G_j$, defined in (\ref{eq:G-j-dis-design}), is usually an amount that becomes controllably smaller as the number of steps $N$ become larger. 
\item[Step 6. (Analyze the approximation ratio and the time complexity)]
Noting that any feasible solution provide a lower bound on the optimal value: $F(x(t_j))\le F(\mathbf{x}^*)$,  and assuming that (\ref{eq:appro-ratio-dis})-(\ref{eq:appro-O-dis}) retain feasibility at all iterations $j=0,\ldots, N-1$, then (\ref{eq:alg-design-dis}) becomes
\begin{eqnarray}
\label{eq:alg-design-dis-G}
E(\mathbf{x}(t_{j+1}))-E(\mathbf{x}(t_j))\ge -G_j^{+}F(\mathbf{x}(t_j))-B_j\ge -G_j^{+}F(\mathbf{x}^*)-B_j.
\end{eqnarray}

From (\ref{eq:appro-alg-dis}), we can bound the $B_j$ in (\ref{eq:B_j}) further: 
\begin{eqnarray}
\nonumber B_j&=&a_{t_{j+1}}\frac{L}{2}\|\mathbf{x}(t_{j+1})-\mathbf{x}(t_j)\|^2=\frac{L}{2}\frac{\left(b_{t_{j+1}}-b_{t_j}\right)^2}{a_{t_{j+1}}}d(t_j)^2\underbrace{\left\|\mathbf{u}(\mathbf{x}(t_j))\right\|^2}_{(\ref{eqn:diamter})}\\
\label{eq:alg-design-dis-B}&\le&\frac{DL}{2}\frac{\left(b_{t_{j+1}}-b_{t_j}\right)^2}{a_{t_{j+1}}}d(t_j)^2
\end{eqnarray}

By using the $B_j$ in (\ref{eq:alg-design-dis-B}) and telescoping the left-hand side of the inequality (\ref{eq:alg-design-dis-G}) over $j=0, \ldots, N-1$, we obtain
\begin{eqnarray*}
\nonumber E(\mathbf{x}(t_N))-E(\mathbf{x}(t_0))&\ge& -F(\mathbf{x}^*) \sum_{j=0}^{N-1}G_j^{+} -\frac{L}{2}\sum_{j=0}^{N-1} \frac{\left(b_{t_{j+1}}-b_{t_j}\right)^2d(t_j)^2}{a_{t_{j+1}}},
\end{eqnarray*}
implying that the performance of the algorithm in (\ref{eq:appro-ratio-dis})-(\ref{eq:appro-O-dis}) is as follows:
\begin{eqnarray}
\nonumber F(\mathbf{x}(t_N))&\ge& \frac{b_{t_N}-b_{t_0}}{a_{t_N}}F(\mathbf{x}^*) + \frac{a_{t_0}}{a_{t_N}}\underbrace{F(\mathbf{x}(t_0))}_{\ge 0}-\frac{\sum_{j=0}^{N-1}G_j^{+}}{a_{t_N}}F(x^*)-\frac{L}{2a_{t_N}}\sum_{j=0}^{N-1} \frac{\left(b_{t_{j+1}}-b_{t_j}\right)^2\left\|\mathbf{u}(\mathbf{x}(t_j))\right\|^2}{a_{t_{j+1}}}\\
\label{eq:con-2-dis}&\ge &\left(\frac{b_{t_N}-b_{t_0}-\sum_{j=0}^{N-1}G_j^{+}}{a_{t_N}}\right)F(\mathbf{x}^*)-\frac{L}{2a_{t_N}}\sum_{j=0}^{N-1} \frac{\left(b_{t_{j+1}}-b_{t_j}\right)^2d(t_j)^2}{a_{t_{j+1}}}
\end{eqnarray}

Note that the first term in the last inequality dictates the approximation ratio, while the second term therein decides the time complexity. 

With $a_t$, $b_t$ and $T$ known from the continuous-time algorithm, the discrete-time implementation has almost the same approximation ratio as specified in the first term in (\ref{eq:con-2-dis}). To obtain the number of iterations $N$ (and hence the time complexity), we choose $N$ and $t_j, j=0,\ldots, N$ to solve a bi-criteria optimization problem: maximize the first term in the right-hand side of (\ref{eq:con-2-dis}), and minimize the second term therein subject to $0=t_0\le\ldots\le t_N=T$, or equivalently,
\begin{eqnarray}
\label{eq:time-comlexity-dis1}&\min\limits_{t_0,\ldots, t_n, N}&\left\{\sum_{j=0}^{N-1}G_j^{+},\sum_{j=0}^{N-1} \frac{\left(b_{t_{j+1}}-b_{t_j}\right)^2d(t_j)^2}{a_{t_{j+1}}} \right\}\\
\label{eq:time-comlexity-dis2}&\text{s.t.} & 0=t_0\le\ldots\le t_N=T
\end{eqnarray}
Our purpose is to estimate the order of the time complexity in the second criteria, while controlling the error of approximation in the first criteria. So we can solve (\ref{eq:time-comlexity-dis1})-(\ref{eq:time-comlexity-dis2}) numerically to help us arriving at a numerically-aided guess on the solution.\qed
\end{description}

Before we illustrate the applications of the formula (\ref{eq:con-2-dis}) in Sections~\ref{subsec:mono-exp1}, \ref{subsec:mono-exp2} and \ref{subsec:con-exp3}, as comparison, we discuss a direct method to design discrete-time algorithm via potential function in the next section. 

\subsubsection{Direct method}\label{eq:subsubsec:direct-method}
An alternate way for designing a discrete-time algorithm is to repeat the four-step process in the first phase directly in the discrete-time domain by replacing the Lyapunov function with the corresponding discretized potential function. 

In the direct method, $a_t,b_t$ and $T$ are unknown. Therefore, the system (\ref{eq:appro-ratio-dis-design})-(\ref{eq:appro-o-dis-design}) may be feasible by carefully choose these quantities. However, allowing $G_j\le 0$ includes this case. The direct method can still use the algorithm designed in (\ref{eq:appro-ratio-dis})-(\ref{eq:appro-O-dis}).

Then at the end, we need to choose $a_{t_j}$, $b_{t_j}$, $T$, $N$, and $t_j$, $j=0,\ldots, N$ to solve a bi-criteria optimization problem: maximize the first term in the right-hand side of (\ref{eq:con-2-dis}), and minimize the second term therein subject to (\ref{eq:appro-ratio-dis})-(\ref{eq:appro-O-dis}), the discrete counterparts $Q$ and $\tilde{P}$ of the constraints of $Q$ and $P$, respectively, and $0=t_0\le\ldots\le t_N=T$:
\begin{eqnarray*}
&\max_{a, b, T, N}& \left\{\frac{b_{t_N}-b_{t_0}-\sum_{j=0}^{N-1}G_j^{+}}{a_{t_N}}, -\sum_{j=0}^{N-1}\frac{\left(b_{t_{j+1}}-b_{t_j}\right)^2d(t_j)^2}{a_{t_{j+1}}}\right\}\\
&\text{s.t.}&(a, b)\in \tilde{Q}\\
&& (a, b, T)\in \tilde{P}\\
& &0=t_0\le\ldots\le t_N=T.
\end{eqnarray*}
This is a nonlinear bi-criteria program. We now establish a connection between this direct method (a simultaneous approach) and the two-phased method (a sequential approach). The approximation ratio usually is the dominant priority. So we solve this bi-criteria problem sequentially via a two-phase approach: first maximizing the first criteria, and then minimizing the second criteria. Our proposed two-phase method is the same except that we solve a continuous-time version of the first phase problem, which is usually easier than its discrete-counterpart. 

However, the direct approach, although usually harder to analyze, offers at least two advantages because we no longer restrict ourselves to the $a_t$ and $b_t$ used in the continuous-time setting. The first one is that we may be able to choose $a_t$ and $b_t$ to remove the extra loss of the term $\sum_{j=0}^{N-1}G_j^{+}$, compared to the two-phase method. The second one is that we may be able to choose $a_t$ and $b_t$ to fine-tune the balance between the approximation ratio and the complexity.


\subsection{Example 1: Nonnegative monotone submodular maximizaiton subject to a general compact convex set}\label{sec:sub-max-con-alg-design-via-L-fun-mono-general-examp1}
For illustration purpose, in the following examples, we closely follow the two-step process as prescribed earlier in Section~\ref{subsec:algo-design-dis-time}. We will choose equal step sizes in all examples, while leaving open the question on whether adpative step sizes can improve the time complexity.

\begin{description}
\item [Step 5.] From the continuous-time problem in Section~\ref{subsec:mono-exp1}, we have  
\[
T=1, a_{t_j}=b_{t_j}=e^{t_j}, j=0,\ldots, N.
\]
The algorithm (\ref{eq:appro-ratio-dis})-(\ref{eq:appro-O-dis}) becomes the following in this example: $\forall j=0, \ldots, N-1$, 
\begin{eqnarray}
\label{eq:appro-ratio-dis-exp1}a_{t_{j+1}}-a_{t_j}&= & b_{t_{j+1}}-b_{t_j}-G_j\\
\label{eq:appro-alg-dis-exp1}\mathbf{x}(t_{j+1})-\mathbf{x}(t_j)&=&\frac{b_{t_{j+1}}-b_{t_j}}{a_{t_{j+1}}}\mathbf{v}(\mathbf{x}(t_j)),
\end{eqnarray}
where
\[
\mathbf{v}(\mathbf{x}(t_j))=\arg\max_{\mathbf{v}\in C}\langle\nabla F(\mathbf{x}(t_j)), \mathbf{v}\rangle.
\]
Note that there is no constraint on the set $Q$, $c(\mathbf{x}(t_j))=1$, $d(t_j)=1$, and $\mathbf{u}(\mathbf{x}(t_j))=\mathbf{v}(\mathbf{x}(t_j))$.

Note first that 
\begin{eqnarray}\label{eq:exp1-G-j}
G_j&=&c(\mathbf{x}(t_j))(b_{t_{j+1}}-b_{t_j})-(a_{t_{j+1}}-a_{t_j})= (b_{t_{j+1}}-b_{t_j})-(a_{t_{j+1}}-a_{t_j}) =0,
\end{eqnarray}
Note further, from $\mathbf{v}(\mathbf{x}(t_j))\in C$, that
\begin{eqnarray}\label{eq:exp1-u}
\left\|\mathbf{u}(\mathbf{x}(t_j))\right\|^2\le D,
\end{eqnarray}
where $D=\max_{\mathbf{x},\mathbf{y}\in C}\|\mathbf{x}-\mathbf{y}\|^2$ is the diameter of $C$.

\item [Step 6.] Choose 
\[
t_j=\frac{j}{N}, j=0,\ldots, N.
\]

Because $G_j^{+}=0$ from (\ref{eq:exp1-G-j}), there is no need to retain feasibility in the algorithm except in the last step because (\ref{eq:con-2-dis}) never invokes $F(x(t_j))\le F(x^*)$. Noting further (\ref{eq:exp1-u}), the approximation ratio and the time complexity in (\ref{eq:con-2-dis}) are as follows:.
\begin{eqnarray}
\nonumber F(\mathbf{x}(t_N))&\ge &\frac{b_{t_N}-b_{t_0}}{a_{t_N}}F(\mathbf{x}^*)-\frac{DL}{2a_{t_N}}\sum_{j=0}^{N-1} \frac{\left(b_{t_{j+1}}-b_{t_j}\right)^2}{a_{t_{j+1}}}\\
\nonumber&=& \left(1-\frac{1}{e}\right)F(\mathbf{x}^*)-\frac{DL}{2a_{t_N}}\sum_{j=0}^{N-1} \frac{\left(e^{t_{j+1}}-e^{t_j}\right)^2}{e^{t_{j+1}}}\\
\nonumber &= &\left(1-\frac{1}{e}\right)F(\mathbf{x}^*)-\frac{DL}{2a_{t_N}}\sum_{j=0}^{N-1} 
e^{t_{j+1}}\left(1-e^{-t_{j+1}+t_j}\right)^2\\
\nonumber&=&\left(1-\frac{1}{e}\right)F(\mathbf{x}^*)-\frac{DL}{2a_{t_N}}\left(1-e^{-\frac{1}{N}}\right)^2\sum_{j=0}^{N-1}e^{\frac{j+1}{N}}\\
\nonumber&\ge &\left(1-\frac{1}{e}\right)F(\mathbf{x}^*)-\frac{DL}{2a_{t_N}}\frac{1}{N^2}(e-1)(N+1)\\
\label{eq:apprx-ratio-example1}&=&\left(1-\frac{1}{e}\right)F(\mathbf{x}^*)-O\left(\frac{DL}{N}\right),
\end{eqnarray}
where the second inequality is from $1-e^{-x}\le x$ with $x=\frac{1}{N}$, $\sum_{j=0}^{N-1}e^{\frac{j+1}{N}}=(e-1)\left(1+\frac{1}{e^{\frac{1}{N}}-1}\right)\le (e-1)(N+1)$, the last equality is from $\sum_{j=0}^{N-2}\frac{1}{N-j-1}=O(\ln N)$, and the last inequality is from  $\ln x\le \sqrt{x}, \forall x\ge 0$.

Therefore from (\ref{eq:apprx-ratio-example1}), we have
\begin{eqnarray*}
F(\mathbf{x}(t_N))&\ge & \left(1-\frac{1}{e}\right)F(\mathbf{x}^*)-\varepsilon,
\end{eqnarray*}
after $O\left(\frac{DL}{\varepsilon}\right)$ iterations. 

As a corollary, if $F$ is the multilinear extension of a submodular set function, then plugging the $L$ in (\ref{eq:mle-L-smooth}) into (\ref{eq:apprx-ratio-example1}) and noting that $D=O(n)$, we obtain 
\[
F(\mathbf{x}(t_N))\ge \left(1-\frac{1}{e}-O\left(\frac{n^3}{N}\right)\right)F(\mathbf{x}^*).
\]
Consequently, 
\[
F(\mathbf{x}(t_N))\ge \left(1-\frac{1}{e}-\varepsilon\right)F(\mathbf{x}^*)
\]
after $O\left(\frac{n^3}{\varepsilon}\right)$ iterations.

\end{description}

\subsection{Example 2: Nonnegative non-monotone submodular maximizaiton subject  to a down-closed compact convex set}\label{sec:sub-max-con-alg-design-via-L-fun-nonmono-down-closed-examp2}


\begin{description}
\item [Step 5.] From the continuous-time problem in Section~\ref{subsec:mono-exp2}, we have  
\[
T=1, a_{t_j}=e^{t_j}, b_{t_j}=t_j, j=0,\ldots, N,
\]
The algorithm (\ref{eq:appro-ratio-dis})-(\ref{eq:appro-O-dis}) becomes the following in this example: $\forall j=0, \ldots, N-1$, 
\begin{eqnarray}
\label{eq:appro-ratio-dis-exp2}a_{t_{j+1}}-a_{t_j}&= & \frac{b_{t_{j+1}}-b_{t_j}}{1-\theta(\mathbf{x}(t_j))}-G_j\\
\label{eq:appro-alg-dis-exp2}\mathbf{x}(t_{j+1})-\mathbf{x}(t_j)&=&\frac{b_{t_{j+1}}-b_{t_j}}{a_{t_{j+1}}(1-\theta(\mathbf{x}(t_j)))}\mathbf{v}(\mathbf{x}(t_j))\\
\label{eq:appro-O-dis-exp2} 1&>&\theta(\mathbf{x}(t_j))\ge \|\mathbf{x}(t_j)\|_{\infty},
\end{eqnarray}
where
\[
 \mathbf{v}(\mathbf{x}(t_j))=\arg\max_{\mathbf{v}\in C, \mathbf{v}\le \mathbf{1}-\mathbf{x}(t_j)}\langle\nabla F(\mathbf{x}(t_j)),\mathbf{v}\rangle.
\]
Note that the last constraint is due to the set $Q$, $c(\mathbf{x}(t_j))=\frac{1}{1-\theta(\mathbf{x}(t_j))}$, $d(t_j))=\frac{1}{1-\theta(\mathbf{x}(t_j))}$ and $\mathbf{u}(\mathbf{x}(t_j))=\mathbf{v}(\mathbf{x}(t_j))$.

We now show that choosing 
\begin{equation}\label{eq:exp2-theta}
1-\theta(\mathbf{x}(t_j))=\frac{1}{c(\mathbf{x}(t_j))}=\frac{1}{a_{t_{j}}},
\end{equation}
leads to $G_j\le 0$ and satisfies the last inequality. 
\begin{proof} First we compute $G_j$ from (\ref{eq:appro-ratio-dis-exp2}):
\begin{eqnarray}
\nonumber G_j&=&\frac{b_{t_{j+1}}-b_{t_j}}{1-\theta(\mathbf{x}(t_j))}-(a_{t_{j+1}}-a_{t_j})\\
\nonumber &=&a_{t_j}(b_{t_{j+1}}-b_{t_j})-(a_{t_{j+1}}-a_{t_j})\\
\label{eq:con-2-dis-exp2}&=&e^{t_j}\left(1+t_{j+1}-t_j- e^{t_{j+1}-t_j}\right)\le 0,
\end{eqnarray}
where the last equality follows from $1+x\le e^x$ with $x=t_{j+1}-t_j$.

From the last inequality, (\ref{eq:appro-alg-dis-exp2}) has an upper bound: $\forall i=1,\ldots, n$, $j=0,\ldots, N-1$, 
\[
\mathbf{x}(t_{j+1})-\mathbf{x}(t_j)=\frac{b_{t_{j+1}}-b_{t_j}}{a_{t_{j+1}}(1-\theta(\mathbf{x}(t_j)))}\mathbf{v}(\mathbf{x}(t_j))\overset{(\ref{eq:con-2-dis-exp2})}{\le} \frac{a_{t_{j+1}}-a_{t_j}}{a_{t_{j+1}}}\mathbf{v}(\mathbf{x}(t_j)),
\]
implying that 
\begin{eqnarray*}
(1-\mathbf{x}_i(t_{j+1}))-(1-\mathbf{x}_i(t_j))&\ge&-\frac{a_{t_{j+1}}-a_{t_j}}{a_{t_{j+1}}}\mathbf{v}_i(\mathbf{x}(t_j))\overset{\mathbf{v}_i\le 1-x_i(t_j)}{\ge} -\frac{a_{t_{j+1}}-a_{t_j}}{a_{t_{j+1}}}(1-x_i(t_j))
\end{eqnarray*}
Hence
\begin{eqnarray*}
1-\mathbf{x}_i(t_{j+1})&\ge & \frac{a_{t_{j}}}{a_{t_{j+1}}}  (1-x_i(t_{j}))\ge \frac{a_{t_{0}}}{a_{t_{j+1}}}=\frac{1}{a_{t_{j+1}}}.
\end{eqnarray*} 
Therefore, 
\begin{equation*}
 1-\|x(t_j)\|_{\infty} \ge \frac{1}{a_{t_j}}=1-\theta(\mathbf{x}(t_j)).
\end{equation*}
So (\ref{eq:appro-O-dis-exp2}) is satisfied by the choice of the $1-\theta(\mathbf{x}(t_j))$
\end{proof}

With (\ref{eq:exp2-theta}) and  $\mathbf{v}(\mathbf{x}(t_j))\in C$, we have
\begin{eqnarray}\label{eq:exp2-u}
d(t_j)^2\left\|\mathbf{u}(\mathbf{x}(t_j))\right\|^2 =d(t_j)^2\left\|\frac{\mathbf{v}(\mathbf{x}(t_j))}{1-\theta(\mathbf{x}(t_j))}\right\|^2=a_{t_{j}}^2\left\|\mathbf{v}(\mathbf{x}(t_j))\right\|^2 \le Da_{t_{j}}^2.
\end{eqnarray}


\item [Step 6.] Choose the step size 
\[
t_j=\frac{j}{N}, j=0,\ldots, N.
\]
Noting that $G_j^{+}=0$ from (\ref{eq:con-2-dis-exp2}), and (\ref{eq:exp2-u}), the approximation ratio and the time complexity in (\ref{eq:con-2-dis}) are as follows:
\begin{eqnarray}
\nonumber F(\mathbf{x}(t_N))&\ge &\frac{b_{t_N}-b_{t_0}}{a_{t_N}}F(\mathbf{x}^*)-\frac{DL}{2a_{t_N}}\sum_{j=0}^{N-1} \frac{\left(b_{t_{j+1}}-b_{t_j}\right)^2a_{t_{j}}^2}{a_{t_{j+1}}}\\
\nonumber&=& \frac{1}{e} F(\mathbf{x}^*)-\frac{DL}{2a_{t_N}}\sum_{j=0}^{N-1} \frac{\left(t_{j+1}-t_j\right)^2e^{2t_j}}{e^{t_{j+1}}}\\
\nonumber&=& \frac{1}{e}F(\mathbf{x}^*)-\frac{DL}{2a_{t_N}N^2}\sum_{j=0}^{N-1}e^{\frac{j-1}{N}} \\
\nonumber&\ge & \frac{1}{e}F(\mathbf{x}^*)-\frac{DL}{2a_{t_N}N^2}\frac{(e-1)N^2}{N+1}\\
\label{eq:apprx-ratio-example2}&=& \frac{1}{e}F(\mathbf{x}^*)-O\left(\frac{DL}{N}\right),
\end{eqnarray}
where the second inequality is from $\sum_{j=0}^{N-1}e^{\frac{j-1}{N}}=\frac{e-1}{e^{\frac{1}{N}}\left(e^{\frac{1}{N}}-1\right)}\le \frac{(e-1)N^2}{N+1}$, the last equality is from $\sum_{j=0}^{N-1}\frac{1}{N-j+1} =O(\ln N)$, and the last inequality is from  $\ln x\le \sqrt{x}, \forall x\ge 0$. Consequently,
\[
F(x(t_N))\ge \frac{1}{e}F(x^*)-\varepsilon
\]
after $T=O\left(\frac{DL}{{\varepsilon}}\right)$ iterations.

As a corollary, if $F$ is the multilinear extension of a submodular set function, then plugging the $L$ in (\ref{eq:mle-L-smooth}) into (\ref{eq:apprx-ratio-example2}) and noting that $D=O(n)$, we obtain
\[
F(x(t_N))\ge \left(\frac{1}{e}-O\left(\frac{n^3}{N}\right)\right)F(x^*).
\]
Consequently,
\[
F(x(t_N))\ge \left(\frac{1}{e}-\varepsilon\right)F(x^*)
\]
 after $T=O\left(\frac{n^3}{{\varepsilon}}\right)$ iterations.

\end{description}

\subsection{Example 3: Nonnegative non-monotone DR-submodular and solvable convex set}\label{subsec:nonmono-examp3}

\begin{description}
\item [Step 5.] From the continuous-time problem in Section~\ref{subsec:con-exp3}, we have  
\[
T=1, a_{t_j}=(1+t_j)^2, b_{t_j}=t_j, j=0,\ldots, N.
\]

The algorithm (\ref{eq:appro-ratio-dis})-(\ref{eq:appro-O-dis}) becomes the following in this example: $\forall j=0, \ldots, N-1$, 
\begin{eqnarray}
\label{eq:appro-ratio-dis-exp3}a_{t_{j+1}}-a_{t_j}&= & \frac{2(b_{t_{j+1}}-b_{t_j})}{1-\theta(\mathbf{x}(t_j))}-G_j\\
\label{eq:appro-alg-dis-exp3}\mathbf{x}(t_{j+1})-\mathbf{x}(t_j)&=&\frac{b_{t_{j+1}}-b_{t_j}}{a_{t_{j+1}}(1-\theta(\mathbf{x}(t_j)))}(\mathbf{v}(\mathbf{x}(t_j))-\mathbf{x}(t_j))\\
\label{eq:appro-O-dis-exp3} 1&>&\theta(\mathbf{x}(t_j))\ge \|\mathbf{x}(t_j)\|_{\infty},
\end{eqnarray}
where
\[
 \mathbf{v}(\mathbf{x}(t_j))=\arg\max_{\mathbf{v}\in C}\langle\nabla F(\mathbf{x}(t_j)), \mathbf{v}\rangle.
\]
Note that the last constraint is due to the set $Q$, $c(\mathbf{x}(t_j))=\frac{2}{1-\theta(\mathbf{x}(t_j))}$, $d(t_j))=\frac{1}{1-\theta(\mathbf{x}(t_j))}$ and $\mathbf{u}(\mathbf{x}(t_j))=\mathbf{v}(\mathbf{x}(t_j))-\mathbf{x}(t_j)$.

We now show that choosing 
\begin{equation}\label{eq:exp3-theta}
1-\theta(\mathbf{x}(t_j))=\frac{2}{c(\mathbf{x}(t_j))}=\frac{1}{\sqrt{a_{t_{j}}}},
\end{equation}
leads to $G_j\le 0$ and satisfies the last inequality. 
\begin{proof} First we compute $G_j$ from (\ref{eq:appro-ratio-dis-exp3}):
\begin{eqnarray}
\nonumber G_j&=&\frac{2(b_{t_{j+1}}-b_{t_j})}{1-\theta(\mathbf{x}(t_j))}-\left(a_{t_{j+1}}-a_{t_j}\right)\\
\nonumber &=&2\sqrt{a_{t_j}}(b_{t_{j+1}}-b_{t_j})-\left(a_{t_{j+1}}-a_{t_j}\right)\\
\nonumber &=&2\sqrt{a_{t_j}}\left(\sqrt{a_{t_{j+1}}}-\sqrt{a_{t_j}}\right)-\left(a_{t_{j+1}}-a_{t_j}\right)\\
\label{eq:con-2-dis-exp3}&=&-\left(\sqrt{a_{t_{j+1}}}-\sqrt{a_{t_j}}\right)^2\le 0.
\end{eqnarray}

Next, from (\ref{eq:appro-alg-dis-exp3}), we have
\begin{eqnarray*}
(1-\mathbf{x}_i(t_{j+1}))-(1-\mathbf{x}_i(t_j))&\ge&-\frac{\sqrt{a_{t_j}} \left(\sqrt{a_{t_{j+1}}}-\sqrt{a_{t_j}}\right)}{a_{t_{j+1}}}(\mathbf{v}(\mathbf{x}(t_j))-\mathbf{x}(t_j))\\
&\overset{\mathbf{v}_i\le 1-x_i(t_j)}{\ge}& -\frac{\sqrt{a_{t_j}} \left(\sqrt{a_{t_{j+1}}}-\sqrt{a_{t_j}}\right)}{a_{t_{j+1}}}(1-x_i(t_j))
\end{eqnarray*}
Hence
\begin{eqnarray*}
1-\mathbf{x}_i(t_{j+1})&\ge &\left(1-\sqrt{\frac{a_{t_j}}{a_{t_{j+1}}}}+\frac{a_{t_j}}{a_{t_{j+1}}} \right) (1-x_i(t_{j+1}))\\
&\ge &\sqrt{\frac{a_{t_j}}{a_{t_{j+1}}}} (1-x_i(t_{j}))\ge \frac{a_{t_{0}}}{\sqrt{a_{t_{j+1}}}}=\frac{1}{\sqrt{a_{t_{j+1}}}}.
\end{eqnarray*} 
Therefore, 
\begin{equation*}
 1-\|x(t_j)\|_{\infty} \ge \frac{1}{\sqrt{a_{t_j}}}=1-\theta(\mathbf{x}(t_j)).
\end{equation*}
So (\ref{eq:appro-O-dis-exp3}) is satisfied by the choice of the $1-\theta(\mathbf{x}(t_j))$
\end{proof}

With (\ref{eq:exp3-theta}) and $\mathbf{v}(\mathbf{x}(t_j)), \mathbf{x}(t_j)\in C$, we have
\begin{eqnarray}\label{eq:exp3-u}
d(t_j)^2\left\|\mathbf{u}(\mathbf{x}(t_j))\right\|^2=d(t_j)^2\left\|\frac{\mathbf{v}(\mathbf{x}(t_j))-\mathbf{x}(t_j)}{1-\theta(\mathbf{x}(t_j))}\right\|^2=a_{t_{j}}\left\|\mathbf{v}(\mathbf{x}(t_j))-\mathbf{x}(t_j)\right\|^2 \le Da_{t_j}.
\end{eqnarray}

\item [Step 6.] Choose 
\[
t_j=\frac{j}{N}, j=0,\ldots, N.
\]

Noting that $G_j^{+}=0$ from (\ref{eq:con-2-dis-exp3}), and (\ref{eq:exp3-u}), the approximation ratio and the time complexity in (\ref{eq:con-2-dis}) are as follows::
\begin{eqnarray}
\nonumber F(\mathbf{x}(t_N))&\ge &\frac{b_{t_N}-b_{t_0}}{a_{t_N}}F(\mathbf{x}^*)-\frac{DL}{2a_{t_N}}\sum_{j=0}^{N-1} \frac{\left(b_{t_{j+1}}-b_{t_j}\right)^2a_{t_j}}{a_{t_{j+1}}}\\
\nonumber&\ge&\frac{1}{4}F(\mathbf{x}^*)-\frac{DL}{2a_{t_N}}\sum_{j=0}^{N-1}\frac{1}{N^2}\\
\label{eq:apprx-ratio-example3}&=&\frac{1}{4}F(\mathbf{x}^*)-O\left(\frac{DL}{N}\right),
\end{eqnarray}
 where the second inequality is from  $\frac{a_{t_j}}{a_{t_{j+1}}}\le 1$ since $a_{t}=(1+t)^2$ is an increasing function. 
Therefore, from (\ref{eq:apprx-ratio-example3})
\begin{eqnarray*}
F(\mathbf{x}(t_N))&\ge & \frac{1}{4} F(\mathbf{x}^*)-\varepsilon,
\end{eqnarray*}
after $O\left(\frac{DL}{\varepsilon}\right)$ iterations. 


Due to the assumption $\mathbf{0}\in C$ in Assumption~\ref{assum:sub-oracle} (1), this result is not in conflict with the hardness result in \cite{Vondrak2013} who shows that no constant approximation algorithm exists with a polynomial number of value oracle calls. If we relax Assumption~\ref{assum:sub-oracle} (1), and initialize the algorithm by an arbitrary feasible solution $\mathbf{x}(0)\in C$, then the approximation ratio becomes
\begin{eqnarray*}
F(\mathbf{x}(t_N))&\ge & \frac{1}{4}\left(1-\|\mathbf{x}(0)\|_{\infty}\right) F(\mathbf{x}^*)-\varepsilon.
\end{eqnarray*}
For example, if we choose $\mathbf{x}(0)=\arg\min_{\mathbf{x}\in C}\|\mathbf{x}\|_{\infty}$, then the approximation ratio becomes
\begin{eqnarray*}
F(\mathbf{x}(t_N))&\ge & \frac{1}{4}\left(1-\min_{\mathbf{x}\in C}\|\mathbf{x}\|_{\infty}\right) F(\mathbf{x}^*)-\varepsilon.
\end{eqnarray*}

As a corollary, if $F$ is the multilinear extension of a submodular set function, then plugging the $L$ in (\ref{eq:mle-L-smooth}) into (\ref{eq:apprx-ratio-example3}) and noting that $D=O(n)$, we obtain
\[
F(x(t_N))\ge \left(\frac{1}{e}-O\left(\frac{n^3}{N}\right)\right)F(x^*).
\]
Consequently,
\[
F(x(t_N))\ge \left(\frac{1}{4}-\varepsilon\right)F(x^*)
\]
 after $O\left(\frac{DL}{\varepsilon}\right)$ iterations.
\end{description}

\section{Discussions}\label{sec:conclusion}

Note that enforcing $a_t=b_t$ brings us back to the convergence analysis in nonlinear programming, although the focus there has been on using Lyapunov function mainly as a proof technique rather than as a compass in algorithm design. 

The proposed framework in this work has the potential to unify many existing results and hence provides new perspectives on how novel algorithms in the literature might have been designed in the first place, as demonstrated in this work by examples from the submodular maximization field. The most important feature of our framework is its transparency which can bring out the major difficulties in the design of algorithms. Therefore, this framework has huge potentials to suggest new insights and ideas to either improve on existing algorithms or attack new problems.

For a given optimization problem, two usually interdependent key factors command the efficiency of the framework: the the parametric form of the bounds on the optimal value and the parametric form of the Lyapunov function. 

One factor is to find useful bounds on the optimal value (upper bounds for maximization problem and lower bounds for minimization problems). These bounds demand the most innovative attention in the framework because they are usually problem-specific. However, general duality theory (such as Fenchel and Lagrangian duality) for non-convex optimization usually offers a systematic scheme on where to look for the bounds (e.g., \cite{toland1978duality,singer2007duality}).

Another factor is the parametric form of the Lyapunov function. The choice of the parametric form is also critical for the design and the consequent analysis of approximation algorithm. While (\ref{eq:Lyapunov-function-continuous-time}) can always serve as the first choice, as explained earlier, other possibilities exist. One open question is whether there exist systematic ways to specify the form of Lyapunov function, or whether we can adopt a non-parametric approach to leave the form of the Lyapunov function unspecified until the end of the process, just like how we treated $a_t$ and $b_t$. Moreover the parametric forms of the bound and the Lyapunov function are not independent. However, we usually take a ``conditional expectation" approach where we fix one first and derive the other conditionally. In this paper we first fix the parametric form of the Lyapunov function and then construct the bound conditionally.

For composite function $F=f+g$ where $f$ is DR-submodular and $g$ is concave, let $\mathbf{x}^*=\arg_{\mathbf{x}\in C} F(\mathbf{x})$ be the optimal solution. The Lyapunov functions take the form 
\[
E(\mathbf{x}(t))=a_tf(\mathbf{x}(t))-b_tf(\mathbf{x}^*)+c_tg(\mathbf{x}(t))-d_tg(\mathbf{x}^*)).
\]
For the DR-submodular function $f$, we can use the earlier bounds on $f(\mathbf{x}^*)$, and for the concave function $g$, there are two possibilities. 
\begin{enumerate}[(i)]
\item If we want to design a bi-criteria algorithm for the problem, namely, finding a solution $\mathbf{x}$ satisfying
\[
f(\mathbf{x})+g(\mathbf{x})\ge \alpha f(\mathbf{x}^*)+\beta g(\mathbf{x}^*),
\]
then for the concave function $g$, Jensen inequality provides a natural bound on $g(\mathbf{x}^*)$:
\[
g(\mathbf{x}^*)\le g(\mathbf{x}(t))+\langle \nabla g(\mathbf{x}(t)), \mathbf{x}^*-\mathbf{x}(t)\rangle
\]
Then almost the same six-step process recovers many existing algorithms and even proposes new algorithms for this type of composite objective functions (e.g.~\cite{mitra2021submodular+}).
\item If we want to study the mirror-decent (projected-gradient as a special case) type of algorithms, then we do not need to use any bound on $g$ as it is usually the Bregman divergence of some simple function (e.g., the Euclidean distance used in the projected-gradient method). Then almost the same six-step process recovers many existing algorithms and even proposes new algorithms for this type of composite objective functions (e.g.~\cite{hassani2017gradient}).
\end{enumerate}

The running examples we used to illustrate the framework are purposely selected because their impacts on the DR-submodular maximization field. In particular, many variants and extensions of these basic problems have been investigated. We believe the transparency of our framework can be very productive in unifying these results and also offer new perspectives to improve existing results. 

Any algorithm for an optimization problem lives in both the time and the space, leading to four possible modeling choices: continuous-time continuous-space, discrete-time continuous-space, continuous-time discrete-space, and discrete-time discrete-space. In this article, we focus on the first two cases. More future work are need to the last two cases as well.

We expect the ideas presented here to be applied to other areas of combination optimization, such as variational inequality problems, games, online problems, stochastic problems, among others. 


 \section*{Acknowledgements} Donglei Du's research is partially supported by the NSERC grant (No. 283106), and NSFC grants (Nos. 11771386 and  11728104). The author would express his appreciation to Qiaoming Han, Jiachen Ju, Yuefang Lian, Zhichen Liu, Dachuan Xu, Xianzhao Zhang, and Yang Zhou for their careful reading of the manuscript with insightful comments which greatly improve the presentation of the paper.


\bibliographystyle{apalike}

\bibliography{Lyapunov}

\begin{thebibliography}{}

\bibitem[Badanidiyuru et~al., 2020]{badanidiyuru2020submodular}
Badanidiyuru, A., Karbasi, A., Kazemi, E., and Vondr{\'a}k, J. (2020).
\newblock Submodular maximization through barrier functions.
\newblock {\em Advances in Neural Information Processing Systems}, 33:524--534.

\bibitem[Bansal and Gupta, 2019]{bansal2019potential}
Bansal, N. and Gupta, A. (2019).
\newblock Potential-function proofs for gradient methods.
\newblock {\em Theory of Computing}, 15(1):1--32.

\bibitem[Bian et~al., 2017]{bian2017continuous}
Bian, A., Levy, K., Krause, A., and Buhmann, J.~M. (2017).
\newblock Continuous dr-submodular maximization: Structure and algorithms.
\newblock In {\em Advances in Neural Information Processing Systems}, pages
  486--496.

\bibitem[Bian et~al., 2020]{bian2020continuous}
Bian, Y., Buhmann, J.~M., and Krause, A. (2020).
\newblock Continuous submodular function maximization.
\newblock {\em arXiv preprint arXiv:2006.13474}.

\bibitem[Buchbinder and Feldman, 2018]{buchbinder2018submodular}
Buchbinder, N. and Feldman, M. (2018).
\newblock Submodular functions maximization problems.
\newblock In {\em Handbook of Approximation Algorithms and Metaheuristics,
  Second Edition}, pages 753--788. Chapman and Hall/CRC.

\bibitem[Buchbinder et~al., 2015]{buchbinder2015tight}
Buchbinder, N., Feldman, M., Seffi, J., and Schwartz, R. (2015).
\newblock A tight linear time (1/2)-approximation for unconstrained submodular
  maximization.
\newblock {\em SIAM Journal on Computing}, 44(5):1384--1402.

\bibitem[Buchbinder et~al., 2009]{buchbinder2009design}
Buchbinder, N., Naor, J.~S., et~al. (2009).
\newblock The design of competitive online algorithms via a primal--dual
  approach.
\newblock {\em Foundations and Trends{\textregistered} in Theoretical Computer
  Science}, 3(2--3):93--263.

\bibitem[Calinescu et~al., 2011]{calinescu2011maximizing}
Calinescu, G., Chekuri, C., Pal, M., and Vondr{\'a}k, J. (2011).
\newblock Maximizing a monotone submodular function subject to a matroid
  constraint.
\newblock {\em SIAM Journal on Computing}, 40(6):1740--1766.

\bibitem[Chen and Luo, 2021]{chen2021unified}
Chen, L. and Luo, H. (2021).
\newblock A unified convergence analysis of first order convex optimization
  methods via strong lyapunov functions.
\newblock {\em arXiv preprint arXiv:2108.00132}.

\bibitem[Clarke, 1990]{clarke1990optimization}
Clarke, F.~H. (1990).
\newblock {\em Optimization and nonsmooth analysis}.
\newblock SIAM.

\bibitem[Cohen-Addad et~al., 2022]{cohen2022improved}
Cohen-Addad, V., Gupta, A., Hu, L., Oh, H., and Saulpic, D. (2022).
\newblock An improved local search algorithm for k-median.
\newblock In {\em Proceedings of the 2022 Annual ACM-SIAM Symposium on Discrete
  Algorithms (SODA)}, pages 1556--1612. SIAM.

\bibitem[Diakonikolas and Orecchia, 2019]{diakonikolas2019approximate}
Diakonikolas, J. and Orecchia, L. (2019).
\newblock The approximate duality gap technique: A unified theory of
  first-order methods.
\newblock {\em SIAM Journal on Optimization}, 29(1):660--689.

\bibitem[Du et~al., 2022]{du2022improved}
Du, D., Liu, Z., Wu, C., Xu, D., and Zhou, Y. (2022).
\newblock An improved approximation algorithm for maximizing a dr-submodular
  function over a convex set.
\newblock {\em arXiv preprint arXiv:2203.14740}.

\bibitem[Durr et~al., 2021]{durr2021non}
Durr, C., Thang, N., Srivastav, A., and Tible, L. (2021).
\newblock Non-monotone dr-submodular maximization over general convex sets.
\newblock In {\em Proceedings of the Twenty-Ninth International Conference on
  International Joint Conferences on Artificial Intelligence}, pages
  2148--2154.

\bibitem[Feldman et~al., 2011]{feldman2011unified}
Feldman, M., Naor, J., and Schwartz, R. (2011).
\newblock A unified continuous greedy algorithm for submodular maximization.
\newblock In {\em 2011 IEEE 52nd Annual Symposium on Foundations of Computer
  Science}, pages 570--579. IEEE.

\bibitem[Feldman and Naor, 2013]{feldman2013maximization}
Feldman, M. and Naor, S. (2013).
\newblock {\em Maximization problems with submodular objective functions}.
\newblock PhD thesis, Computer Science Department, Technion.

\bibitem[Frank et~al., 1956]{frank1956algorithm}
Frank, M., Wolfe, P., et~al. (1956).
\newblock An algorithm for quadratic programming.
\newblock {\em Naval research logistics quarterly}, 3(1-2):95--110.

\bibitem[Gronwall, 1919]{gronwall1919note}
Gronwall, T.~H. (1919).
\newblock Note on the derivatives with respect to a parameter of the solutions
  of a system of differential equations.
\newblock {\em Annals of Mathematics}, pages 292--296.

\bibitem[Hassani et~al., 2017]{hassani2017gradient}
Hassani, H., Soltanolkotabi, M., and Karbasi, A. (2017).
\newblock Gradient methods for submodular maximization.
\newblock {\em arXiv preprint arXiv:1708.03949}.

\bibitem[Jacimovic and Geary, 1999]{jacimovic1999continuous}
Jacimovic, M. and Geary, A. (1999).
\newblock A continuous conditional gradient method.
\newblock {\em Yugoslav journal of operations research}, 9(2):169--182.

\bibitem[Lyapunov, 1992]{Lyapunov1992general}
Lyapunov, A.~M. (1992).
\newblock The general problem of the stability of motion.
\newblock {\em International journal of control}, 55(3):531--534.

\bibitem[Mitra et~al., 2021]{mitra2021submodular+}
Mitra, S., Feldman, M., and Karbasi, A. (2021).
\newblock Submodular+ concave.
\newblock {\em Advances in Neural Information Processing Systems}, 34.

\bibitem[Nemirovskij and Yudin, 1983]{nemirovskij1983problem}
Nemirovskij, A.~S. and Yudin, D.~B. (1983).
\newblock Problem complexity and method efficiency in optimization.

\bibitem[Polyak, 1987]{polyak1987introduction}
Polyak, B.~T. (1987).
\newblock Introduction to optimization. optimization software.
\newblock {\em Inc., Publications Division, New York}, 1:32.

\bibitem[Rockafellar, 1970]{rockafellar1970convex}
Rockafellar, R. (1970).
\newblock Convex analysis.

\bibitem[Singer, 2007]{singer2007duality}
Singer, I. (2007).
\newblock {\em Duality for nonconvex approximation and optimization}.
\newblock Springer Science \& Business Media.

\bibitem[Su et~al., 2016]{su2016differential}
Su, W., Boyd, S., and Candes, E.~J. (2016).
\newblock A differential equation for modeling nesterov's accelerated gradient
  method: Theory and insights.
\newblock {\em The Journal of Machine Learning Research}, 17(1):5312--5354.

\bibitem[Toland, 1978]{toland1978duality}
Toland, J.~F. (1978).
\newblock Duality in nonconvex optimization.
\newblock {\em Journal of Mathematical Analysis and Applications},
  66(2):399--415.

\bibitem[Vondr{\'a}k, 2013]{Vondrak2013}
Vondr{\'a}k, J. (2013).
\newblock Symmetry and approximability of submodular maximization problems.
\newblock {\em SIAM Journal on Computing}, 42(1):265--304.

\bibitem[Wibisono et~al., 2016]{wibisono2016variational}
Wibisono, A., Wilson, A.~C., and Jordan, M.~I. (2016).
\newblock A variational perspective on accelerated methods in optimization.
\newblock {\em proceedings of the National Academy of Sciences},
  113(47):E7351--E7358.

\bibitem[Wilson, 2018]{wilson2018lyapunov}
Wilson, A. (2018).
\newblock {\em Lyapunov arguments in optimization}.
\newblock University of California, Berkeley.

\bibitem[Wilson et~al., 2021]{wilson2021Lyapunov}
Wilson, A.~C., Recht, B., and Jordan, M.~I. (2021).
\newblock A lyapunov analysis of accelerated methods in optimization.
\newblock {\em Journal of Machine Learning Research}, 22(113):1--34.

\end{thebibliography}

\end{document}